\documentclass[a4paper,leqno,11pt]{amsart}
\usepackage[
top=30truemm,
 bottom=30truemm,
left=30truemm,
right=30truemm]{geometry}
\usepackage{amsfonts, amssymb, amsmath, verbatim, amsthm, mathrsfs, amscd, enumerate, ascmac, fancyhdr, caption, hyperref, framed,ulem,subfigure}
\usepackage{bbm}%mathbbm{1}
\usepackage{eucal}%fancy \mathcal
\usepackage{tikz}
\usetikzlibrary{shapes.geometric}
\usetikzlibrary{arrows.meta,arrows}
\usepackage{graphicx} % Required for inserting images
\usepackage{pgf, tikz}
\usepackage{pgfplots}
\pgfplotsset{compat=1.18}
\usepackage{float}
\usepackage{pdfpages}

\usepackage[framemethod=tikz]{mdframed}
\mdfsetup{linecolor=blue,backgroundcolor=gray!10,roundcorner=10pt,leftmargin=2pt,innerleftmargin=2pt,innerrightmargin=2pt}
\usepackage{url}

\usepackage{xcolor}
\usepackage{tikz}
\usetikzlibrary{positioning,intersections, calc,arrows,decorations.markings}

\newtheorem{theorem}{Theorem}[section]
\newtheorem{lemma}[theorem]{Lemma}
\newtheorem{proposition}[theorem]{Proposition}
\newtheorem*{theorem*}{Theorem}
\newtheorem{conjecture}{Conjecture}
\newtheorem*{conjecture*}{Conjecture}
\newtheorem{question}[theorem]{Question}

\newtheorem{corollary}[theorem]{Corollary}

\newtheorem*{claim*}{Claim}
\theoremstyle{definition}

\newtheorem*{goal*}{Goal}

\numberwithin{equation}{section}
\numberwithin{theorem}{section}

 \DeclareMathOperator{\dist}{dist}
 
\newtheorem{remark}[theorem]{Remark}

\def\supp{{\rm supp}}

\newcommand{\BL}{\operatorname{BL}}
\newcommand{\ABL}{\operatorname{ABL}}

\title[Stability for the adjoint Finner inequality]{Stability for the adjoint Finner inequality and improved lower bounds on the $k$-plane transform}
\author{Jonathan Bennett and Sewook Oh}\thanks{The first author was supported by the Engineering and Physical Sciences Research Council Grant EP/W032880/1. The second author was supported by KIAS Individual Grant SP089101 and a Global University Project grant funded by the GIST in 2026. Both authors thank Shohei Nakamura for numerous helpful discussions on this and related work.}
\address{Jonathan Bennett, School of Mathematics, University of Birmingham, Edgbaston, B15 2TT, UK}
\email{j.bennett@bham.ac.uk}
\address{Sewook Oh, Department of Mathematical Sciences, Gwangju Institute of Science and Technology, Republic of Korea}
\email{sewookoh@gist.ac.kr}
\date{October 2026}
\begin{document}

\begin{abstract}
    We establish stability properties of near-extremisers for the Loomis--Whitney inequality and its adjoint, presenting our results in the general framework of Finner. In particular, we show that a function nearly achieves equality in the adjoint Finner inequality if and only if it is near to being an indicator function of a cartesian product. As an application, we obtain new  lower bounds for the classical $k$-plane transform on $L^p(\mathbb{R}^d)$, modestly improving those obtained recently by Tao and the first author. We also present an alternative approach to such improvements based on stability in the setting of the Brascamp--Lieb inequality and the theory of fusion frames. We conclude with an interpretation of our $k$-plane transform estimates as adjoint forms of the continuous Brascamp--Lieb inequalities.
\end{abstract}
\maketitle

\section{Introduction} 
The Loomis--Whitney inequality, originating in \cite{LW}, is a classical and elementary inequality with important applications in many areas, including harmonic analysis, partial differential equations, additive combinatorics and information theory. It has found a range of useful generalisations and extensions, notably within the Brascamp--Lieb and related frameworks in harmonic analysis; see for example \cite{Guth10, BB21, Zh22}. 

For $d\geq 2$ and $1\leq j\leq d$, let $P_j$ denote the orthogonal projection from $\mathbb R^d$ to $\langle e_j\rangle^\perp$, where $\{e_j: 1\leq j\leq d\}$ is the standard orthonormal basis of $\mathbb R^d$. The Loomis--Whitney inequality states that 
\begin{equation}\label{lw}
    \int_{\mathbb R^d} \prod_{j=1}^d g_j(P_jx)\mathrm{d}x\le\prod_{j=1}^d\|g_j\|_{L^{d-1}(\langle e_j\rangle^\perp)}
\end{equation} 
for all nonnegative measurable functions $g_j: \langle e_j\rangle^\perp\rightarrow \mathbb R_+$. This is a sharp inequality, and functions $(g_j)$ that achieve equality in \eqref{lw}, referred to as \textit{extremisers}, may be characterised explicitly, as we clarify shortly. 
The Loomis--Whitney inequality is a \textit{measure-theoretic} (as opposed to functional) inequality by origin, related to isoperimetry. In particular, setting $g_j=\mathrm{1}_{P_j(A)}$, one obtains the inequality
\begin{equation}
    \label{LWgeometric}
|A|^{d-1}\le \prod_{j=1}^d|P_j(A)|,
\end{equation}
valid for all measurable $A\subset\mathbb R^d$; here and throughout $|E|$ denotes the Lebesgue measure of a subset $E$ of Euclidean space.

The primary focus of this paper is a certain \textit{adjoint} form of the Loomis--Whitney inequality \eqref{lw} introduced recently by Tao and the first author in \cite{BT}. In its simplest form this states that
\begin{equation}\label{adjlw}
\|f\|_{L^p(\mathbb{R}^d)}\leq\prod_{j=1}^d\|(P_j)_*f\|_{L^q(\langle e_j\rangle^\perp)}^{\frac{1}{d}},
\end{equation}
where $(P_j)_*f:\langle e_j\rangle^\perp\rightarrow\mathbb{R}_+$ is the \textit{$j$th marginal} of a nonnegative function $f\in L^p(\mathbb{R}^d)$, effectively obtained by integrating $f(x)$ in the $j$th component $x_j$ of $x$. Here the exponents $p$ and $q$, which necessarily lie in the quasi-Banach range $0<p,q\leq 1$, are related by the elementary scaling relation 
$\frac{d}{p}-\frac{d-1}{q}=1$. Taking logarithms one may interpret \eqref{adjlw} as an extension of Shearer's entropy subadditivity to the setting of  R\'enyi entropies; see the forthcoming Remark \ref{renyisub}.

Just as for \eqref{lw}, the extremisers of \eqref{adjlw} have been characterised explicitly -- in the latter case they are simply positive constant multiples of indicator functions of cartesian products of subsets of $\mathbb{R}$; see \cite[Theorem 2.1]{BT}. In such situations it is natural (and common) to ask questions of stability. In our case we ask: 
\begin{quotation}
\textit{If a (nonnegative) function $f$ achieves near-equality in \eqref{adjlw}, then is it close to an extremiser in some concrete sense?}
\end{quotation}

In \cite{finner} Finner established a natural generalisation of the Loomis--Whitney inequality \eqref{lw} to the setting of product measure spaces, and this, and its adjoint, will be the natural setting for our stability results. We are not the first to tackle stability questions in this space. In \cite{EFKY} Ellis, Friedgut, Kindler and Yehudayoff considered the stability of the measure-theoretic form of Finner's inequality, proving among other things that if a set $A$ nearly attains equality in \eqref{LWgeometric} then its symmetric difference with an extremising set must have small measure. In order to prove a stability statement for the adjoint Finner inequality (the forthcoming Theorem \ref{thm:adfstable2}) we shall need
%Recently in \cite{BT}, Tao and the first author established an \textit{adjoint form} of Finner's inequality, and used it to prove some unexpected lower bounds on the classical $k$-plane transform in the scale of $L^p$ spaces for $0<p\leq 1$.
%The main purpose of this paper is to investigate the stability of near-extremisers for the adjoint form of Finner's inequality, and establish that inputs nearly achieving equality are near to actual extremisers, with respect to suitable metrics . 
a stability statement for the general functional (as opposed to measure-theoretic) form of Finner's original inequality. Such a statement has recently been provided by Cook and Dembo \cite{CD} with respect to an $L^1$-based metric. We establish a stronger $L^2$-based form of this with the optimal stability exponent, consistent with the measure-theoretic statements in \cite{EFKY}; see Remarks \ref{Rem:CD} and \ref{Rem:EFKY}.

As an application of our stability results for the adjoint inequality, we modestly improve the lower bounds for the $k$-plane transform established in \cite{BT}. This is motivated by Question 10.4 in \cite{BT}, and it is interesting to compare our broad approach here with that of Kova\v{c} and Rogers \cite{KR} in their recent answer to Question 10.1 of \cite{BT}. An additional ingredient in our argument is the classical Brascamp--Lieb--Luttinger rearrangement inequality, and it is interesting to compare this aspect with a similar step in work of Milman and Yehudayoff \cite{MY}. Finally, we describe an alternative approach to such $k$-plane transform estimates based on the notion of \textit{Gaussian--stability} for Brascamp--Lieb inequalities explored recently in \cite{BT,BC:stable}, and the theory of tight fusion frames.

Some wider contextual remarks are in order. Finner's inequality falls under the increasingly broad umbrella of Brascamp--Lieb theory, in which there is a growing literature relating to stability. This includes numerous works of Christ in the setting of inequalities for multilinear forms, such as Young's convolution inequality; see \cite{BC:stable} and the references there. We also point to the stability results for the Pr\'ekopa--Leindler and Borell--Brascamp--Lieb inequalities in work of Figalli, van Hintum and Tiba \cite{FHT}. 
We  refer the reader to these sources for further results and historical discussion relating to stability.

%The adjoint Brascamp--Lieb inequalities of \eqref{BT}, of which \eqref{adjlw} is an example, is one of several ``dual" forms of the general Brascamp--Lieb inequalities, beginning with the reverse Brascamp--Lieb inequalities of Barthe \cite{Barthe}. The Pr\'ekopa--Leindler and Borell--Brascamp--Lieb inequalities mentioned above are examples

The remainder of this introduction is organised as follows. In Section \ref{intf} we introduce Finner's inequality and state a stability result for  its general functional form. In Section \ref{intadjf} we introduce the adjoint Finner inequality and state our main result: Theorem \ref{thm:adfstable2}. Finally, in Section \ref{intk} we describe our applications to the theory of the $k$-plane transform, our main conclusion being Theorem \ref{thm:main}.

\subsection{Stability for Finner's inequality}\label{intf} In \cite{finner}, Finner established a broad generalisation of the Loomis--Whitney inequality \eqref{lw} to the setting of arbitrary product measure spaces.

Let $S_1,\hdots,S_m\subseteq \{1,\hdots,d\}$ and $1\le q_1,\hdots,q_m<\infty$ be such that
\begin{equation}
  \label{finnerrelation}  
\sum_{j=1}^m\frac{1}{q_j}1_{S_j}(i)
=1
\end{equation}
for each $1\leq i\leq d$, and let $(\Omega_i,\mu_i)$ be measure spaces for $i=1,\hdots,d$.
For a subset $S\subseteq \{1,\hdots,d\}$, define the product measure space $(\Omega_{S},\mu_{S})$  by
\[
\Omega_{S}=\prod_{i\in S}\Omega_i, \quad \mu_S=\bigotimes_{i\in S}\mu_i.
\]
We denote by $(\Omega,\mu)$ the full product measure space corresponding to $S=\{1,\hdots,d\}$, and write $x=(x_1,\hdots,x_d)\in \Omega_1\times\cdots\times\Omega_d=\Omega$.
Defining $P_{j}:\Omega\rightarrow \Omega_{S_j}$ to be the coordinate
%\footnote{Changed orthogonal to coordinate; I see} 
projection onto $\Omega_{S_j}$, Finner showed that
\begin{equation}\label{f}
\int_{\Omega}\prod_{j=1}^mg_j(P_{j}x)\mathrm{d}\mu(x)\le \prod_{j=1}^m\|g_j\|_{L^{q_j}(\Omega_{S_j})}
\end{equation}
for all (nonnegative) functions $g_j\in L^{q_j}(\Omega_{S_j})$, where $$\|g_j\|_{L^{q_j}(\Omega_{S_j})}^{q_j}:=\int_{\Omega_{S_j}}|g_j|^{q_j}\mathrm d\mu_{S_j}.$$ In the special case where $\Omega_j=\mathbb R$, the measure $\mu_j$ is Lebesgue measure, $m=d$, and $S_j=\{1,\hdots,d\}\setminus \{j\}$ for all $j$, the equations \eqref{finnerrelation} have the unique solution $q_1=\cdots=q_d=d-1$, and \eqref{f} becomes the Euclidean Loomis--Whitney inequality \eqref{lw}. If $S_j=\{1,\hdots,d\}$ for all $j$, then \eqref{f} is simply the $m$-linear H\"older inequality; we note that there is no loss of generality in supposing that $d=1$ in this case. A fully degenerate situation is when the sets $S_1,\hdots,S_m$ partition $\{1,\hdots,d\}$, forcing $q_1=\cdots=q_m=1$, so that \eqref{f} is an instance of Fubini's identity. We remark that if an exponent $q_j=1$, there is no loss of generality in supposing that the corresponding $S_j$ is a singleton; one may simply aggregate the $\Omega_i$ for $i\in S_j$.

The sets $S_1,\hdots,S_m$ cover $\{1,\hdots,d\}$ by \eqref{finnerrelation}, and so the sets $$M_i:=\{j:i\in S_j\}$$ are nonempty for each $1\leq i\leq d$.
Under the assumption that
\begin{equation}
    \label{assumptionM}
M_{i_1}\neq M_{i_2}
\end{equation}
for $i_1\neq i_2$, which evidently includes the Loomis--Whitney case, Finner \cite{finner} verified that equality occurs in \eqref{f} if and only if there exist constants $c_j>0$ and nonnegative functions $h_i:\Omega_i\rightarrow \mathbb R_+$ such that
\[
g_j(P_jx)^{q_j}=c_j\prod_{i\in S_j}h_i(x_i)
\]
almost everywhere.
We note that by aggregating $\Omega_{i_1}$ and $\Omega_{i_2}$ if $M_{i_1}=M_{i_2}$, we can always assume that \eqref{assumptionM} holds.
%The assumption \eqref{assumptionM} is necessary, in the sense that without \eqref{assumptionM}, one cannot guarantee that extremisers are given by products of functions on $\Omega_i$. 
Under the assumption \eqref{assumptionM}, we begin by proving the following stability statement for the functional inequality \eqref{f}. As we clarify shortly, this sharpens a similar result of Cook and Dembo \cite{CD}, matching the stability exponent of Ellis, Friedgut, Kindler and Yehudayoff \cite{EFKY} in the case of characteristic functions.
\begin{theorem}[Stability for Finner]\label{thm:stablefinner}
    Let $d,m\ge1$, $1\le q_1,\cdots,q_m<\infty $ and $S_1,\hdots, S_m\subseteq \{1,\hdots,d\}$ satisfy \eqref{finnerrelation} and \eqref{assumptionM}. There exists a constant $C>0$ such that if nonnegative functions $g_j:\Omega_{S_j}\rightarrow \mathbb R_+$ satisfy 
    \begin{equation}
        \label{finnerstablecond}
    \int_{\Omega}\prod_{j=1}^mg_j(P_{j}x)\mathrm{d}\mu(x)\ge (1-\delta)\prod_{j=1}^m\|g_j\|_{L^{q_j}(\Omega_{S_j})},
    \end{equation}
    for $0\le\delta<1$,
    then there exist constants $c_j>0$ for $j=1,\hdots, m$ and nonnegative functions $h_i:\Omega_i\rightarrow \mathbb R_+$ for $i=1,\hdots,d$ such that
    \begin{equation}
        \label{finnerstablest}
    \|g_j^{\frac{q_j}{2}}-H_j^{\frac 12}\|_{L^2(\Omega_{S_j})}\le C \delta^{\frac 12} \|g_j\|_{L^{q_j}(\Omega_{S_j})}^{\frac{q_j}{2}}
    \end{equation}
    for all $j=1,\hdots, m$, where $H_j(P_jx)=c_j\prod_{i\in S_j}h_{i}(x_i)$. Furthermore, one may take each $h_i$ as
    \begin{equation}
        \label{hifinner}
    h_i(x_i)=\int g_j(P_jx)^{q_j}\mathrm{d}\mu_{S_j\setminus \{i\}}
    \end{equation}
    for some $j=j(i)\in M_i$; in cases that $S_j= \{i\}$, the formula \eqref{hifinner} should be interpreted as
   $
    h_i(x_i)=g_j(P_jx)^{q_j}.
   $
   \end{theorem}
\begin{remark}[Relation with \cite{CD}]\label{Rem:CD}
In \cite{CD}, Cook and Dembo established a form of stability for Finner's inequality. Their result can be interpreted as showing that \eqref{finnerstablecond} implies  
   \begin{equation}
\nonumber
    \|g_j^{{q_j}}-H_j\|_{L^1(\Omega_{S_j})}\leq C\delta^{c} \|g_j\|_{L^{q_j}(\Omega_{S_j})}^{{q_j}},
    \end{equation}
    for some $c>0$. The identification of the optimal stability exponent $c$ is left as a question in \cite{CD}, and Theorem \ref{thm:stablefinner} provides an answer to this (the optimal exponent is $c=1/2$ in all interesting cases; see Section \ref{sec:sharpness}), although one might also be interested in the optimal value of the multiplicative constant $C$ -- something that we do not consider here.
%    \footnote{JB: I suspect a bit of work is needed here if we want to claim this.  I think $1/2$ isn't always optimal, but it is if you want a single exponent that works for all Finner data. Fubini illustrates this, no? Maybe the optimal constant is $1/2$ in all other cases! Maybe the argument of Proposition 3.4 in \cite{BC:stable} can do this. Maybe you see this for some other reason...;}
    \end{remark}

\begin{remark}[Relation with \cite{EFKY}]\label{Rem:EFKY}
 Theorem \ref{thm:stablefinner} recovers (up to the identification of an explicit multiplicative constant $C$) the stability result for the measure-theoretic form of Finner's inequality,
    \[
    \mu(A)\le \prod_{j=1}^m\mu_{S_j}(P_j(A))^{\frac 1{q_j}},
    \]
    established in \cite{EFKY}.
 Theorem \ref{thm:stablefinner} shows that  for $\delta>0$, the condition
    \begin{equation}
        \label{geofinnercond}
    \mu(A)\ge (1-\delta)\prod_{j=1}^m\mu_{S_j}(P_j(A))^{\frac 1{q_j}}
    \end{equation}
    implies 
    \begin{equation}\label{mtb}
    \mu(A\triangle \mathcal P)\le C\delta \mu(A)
    \end{equation}
    for some
     cartesian product set $\mathcal P$; here $A\triangle \mathcal P=(A\setminus \mathcal P)\cup (\mathcal P\setminus A)$ is the symmetric difference of $\mathcal{P}$ and $A$; see Section \ref{Sect:stabmeas}.
It may be of interest to note that \eqref{mtb} does not follow from the weaker $L^1$-based functional stability statement discussed in Remark \ref{Rem:CD}, since the optimal exponent there is $c=1/2$, rather than $1$; see Section \ref{sec:sharpness}.
\end{remark}

\subsection{Stability for the adjoint Finner inequality}\label{intadjf}
With the notation of the previous subsection, let $(P_j)_*f:\Omega_{S_j}\rightarrow \mathbb{R}_+$ denote the pushforward of a function $f:\Omega\rightarrow\mathbb{R}_+$ by $P_j$, defined by duality by
\begin{equation}\label{pbpf}
\int_{\Omega_{S_j}} (P_j)_*f(y)g(y)\mathrm{d}\mu_{S_j}(y)=\int_{\Omega} f(x)g(P_jx)\mathrm{d} \mu(x),\end{equation}
and explicitly by the marginal formula
$$
(P_j)_*f(y)=\int_{P_j^{-1}(y)}f\;\mathrm{d}\mu_{S_j^c};
$$
the latter follows from Fubini's theorem, and we clarify that $S_j^c$ is the complement of $S_j$ in $\{1,\hdots,d\}$.
The adjoint of Finner's inequality \eqref{f}, as defined in \cite{BT}, takes the form 
\begin{equation}\label{af}
\|f\|_{L^p(\Omega)}\le \prod_{j=1}^m\|(P_j)_*f\|_{L^{p_j}(\Omega_{S_j})}^{\theta_j},
\end{equation}
where  $\theta_1,\hdots,\theta_m>0$ satisfy the homogeneity condition \begin{equation}\label{sum1}\theta_1+\cdots+\theta_m=1,\end{equation} and the exponents $0<p,p_1,\hdots,p_m\leq 1$ satisfy the compatibility condition
\begin{equation}
    \label{relation:thetajpj}
    \frac{1}{q_j}\left(1-\frac 1p\right)=\theta_j\left(1-\frac 1{p_j}\right)
\end{equation}
for each $1\leq j\leq m$, where $q_1,\hdots,q_m$ satisfy \eqref{finnerrelation}. Evidently the particular adjoint Loomis--Whitney inequality \eqref{adjlw} stated earlier corresponds to the case where $m=d$, $\Omega_j=\mathbb R$, the measure $\mu_j$ is Lebesgue measure, $S_j=\{1,\hdots,d\}\setminus \{j\}$,  $q_j=d-1$, $\theta_j=1/d$, and the common exponents $p_j=q$ are given by \eqref{relation:thetajpj}. 

If $p=1$ then necessarily $p_j=1$ for all $j$ by \eqref{relation:thetajpj}, and \eqref{af} becomes an identity by Fubini's theorem. However, if $p<1$ and the nondegeneracy condition \eqref{assumptionM} is satisfied, a straightforward adaptation of the proof of \cite[Theorem 2.1]{BT}, which characterises the extremisers in the Loomis--Whitney case, shows that the extremisers of the adjoint Finner inequality \eqref{af} are precisely the functions of the form
\[
f=c1_{E_1\times\cdots\times E_d},
\]
where $E_i\subseteq\Omega_i$ has positive measure for each $i$, and $c$ is a positive constant.

Our main theorem is the following associated stability statement.
\begin{theorem}[Stability for adjoint Finner]\label{thm:adfstable2}
   Suppose $0<p<1$ and that the sets $S_j$ and exponents $q_j, p_j, \theta_j$ satisfy \eqref{finnerrelation}, \eqref{assumptionM}, \eqref{sum1} and \eqref{relation:thetajpj}. Then, for all $\varepsilon>0$ there exists $\delta>0$ such that if 
    \begin{equation}
        \label{ineq:stableadfcond}
    \|f\|_{L^p(\Omega)}\ge (1-\delta) \prod_{j=1}^m\|(P_j)_*f\|_{L^{p_j}(\Omega_{S_j})}^{\theta_j},
    \end{equation}
    then there exist a constant $c>0$ and sets $E_i\subseteq \Omega_i$ of positive measure such that
    \begin{equation}\label{notgoodenough}
    \|f-c1_{E_1\times\cdots \times E_d}\|_{L^p(\Omega)}\le  \varepsilon\|f\|_{L^p(\Omega)}
\end{equation} 
and
    \begin{equation}
        \label{ineq:pjfpj1}
    \|(P_j)_*(f-c1_{E_1\times\cdots \times E_d})\|_{L^{p
    _j}(\Omega_{S_j})}\le \varepsilon \|(P_j)_*f\|_{L^{p_j}({\Omega_{S_j}})}
    \end{equation}
    for each $1\leq j\leq m$.
\end{theorem}

Together, the conclusions \eqref{notgoodenough} and \eqref{ineq:pjfpj1} provide a complete (qualitative) description of near extremisers for the adjoint Finner inequality \eqref{af}. To be quantitative one would want to clarify the relationship between $\varepsilon$ and $\delta$ as we do in the statement of Theorem \ref{thm:stablefinner}. An inspection of our proof of Theorem \ref{thm:adfstable2} reveals that one may take $\delta$ to be a certain power of $\varepsilon$, although this power seems unlikely to be optimal; it may be interesting to identify the best-possible exponent in terms of the underlying data.

\begin{remark}
We note that the estimate \eqref{notgoodenough} alone does not guarantee that $f$ is a near extremiser due to a certain instability of the marginal operators $(P_j)_*$ on $L^p$ for $p<1$; for example, in the case where $\Omega_i=\mathbb{R}$ and $\mu_i$ is Lebesgue measure for each $i$, if
    \[
g(x)=\prod_{i=1}^d|x_i|^{-\frac 1{p+}}1_{[-1,1]}(x_i)
    \]
    for some $p<p_+<1$, then $\|f+\varepsilon g\|_{L^p(\mathbb{R}^d)}$ is finite, yet
%the expression $\prod_{j=1}^m\|(P_j)_*f\|_{L^{p_j}(\Omega_{S_j})}^{\theta_j}$ is finite, whereas
    \[
    \prod_{j=1}^m\|(P_j)_*(f+\varepsilon g)\|_{L^{p_j}(\Omega_{S_j})}^{\theta_j}
    \]
    is infinite for any $f\in L^p(\mathbb{R}^d)$ satisfying \eqref{ineq:stableadfcond}, for all $\varepsilon>0$; here we are excluding the degenerate cases where either $p=1$ or $S_1=\cdots=S_m=\{1,\hdots,d\}$ (H\"older's inequality). This makes \eqref{ineq:stableadfcond} unstable under perturbations with respect to the $L^p$ (quasi-) metric.  The additional inequalities \eqref{ineq:pjfpj1} account for this. It is worth noting that \eqref{notgoodenough} is not a direct consequence of combining \eqref{ineq:pjfpj1} with the adjoint Finner inequality \eqref{af}. This is because $f-c1_{E_1\times\cdots \times E_d}$ is not in general single-signed, as required by adjoint Brascamp--Lieb inequalities; see \cite{BT}. 
\end{remark}   
\begin{remark}[An entropic interpretation]\label{renyisub}
	As observed in \cite[Remark 1.13]{BT}, the inequality \eqref{af} may be written as 
	\begin{equation}\label{rensub}
	\mbox{h}_p(f)\leq\sum_{j=1}^m q_j^{-1}\mbox{h}_{p_j}((P_j)_*f)
	\end{equation}
	where $\mbox{h}_p(f):=\frac{p}{1-p}\log\|f\|_p$ is the \textit{R\'enyi entropy} of $f$. Noting that $$\mbox{h}_p(f)\rightarrow \mbox{h}(f)=-\int f\log f$$ as $p\rightarrow 1$, this may be viewed as an extension of Shearer's inequality for the Shannon entropy to the full R\'enyi scale; we note that Shearer's inequality sits within the broader entropic Brascamp--Lieb framework of Carlen and Cordero-Erausquin \cite{CC}, and we refer to \cite[Remark 1.13]{BT} for the associated generalisation of \eqref{rensub}. As will be expected, Theorem \ref{thm:adfstable2} may be interpreted as a stability result for \eqref{rensub}: For all $N\ge1$, there exists $\delta>0$ such that if $$\mbox{h}_p(f)\ge-\delta+ \sum_{j=1}^m q_j^{-1}\mbox{h}_{p_j}((P_j)_*f),$$ then there exist $E_i\subseteq\Omega_i$ of positive measure for each $i$, and $c>0$ such that
%    \footnote{added; checked}
$$\mbox{h}_p(f-c1_{E_1\times\cdots\times E_d})\leq\mbox{h}_p(f)-N$$
    and
    $$\mbox{h}_{p_j}((P_j)_*(f-c1_{E_1\times\cdots\times E_d}))\le \mbox{h}_{p_j}((P_j)_*f) -N$$ for each $1\leq j\leq m$. 
\end{remark}
\begin{remark}
    The adjoint Finner inequality \eqref{af}, or rather its generalisation in \cite{BT}, is one of several ``dual" forms of the Brascamp--Lieb inequality that have emerged over the years, beginning with the reverse Brascamp--Lieb inequality of Barthe \cite{Barthe} -- a setting that has also seen progress on stability; see \cite{FHT}.
    We point to \cite{CC,BW,CDP,CHV,Liak, BCo} and the references there for some further examples of dual forms.
    %Further examples of ``dual" forms include the entropic Brascamp--Lieb inequality of Carlen and Cordero-Erausquin \cite{CC} (see Remark \ref{renyisub}), the inverse Brascamp--Lieb inequality of Barthe and Wolff \cite{BW} and Chen, Dafnis, and Paouris \cite{CDP}, and the Fourier Brascamp--Lieb inequality developed by Cowling and the first author in \cite{BCo}, the multilinear duality statements of  \cite{Guth} and Carbery, Hanninen and Valdimarsson \cite{CHV}, the Fourier Brascamp--Lieb inequality developed by Cowling and the first author \cite{BCo}
\end{remark}

\begin{comment}
It is perhaps plausible that a full characterisation of near extremisers is available with respect to a stronger metric, such as $(f,g)\mapsto\sum_j\|(P_j)_*(f-g)\|_{p_j}$.
    
\end{comment}

We now turn to our applications of Theorem \ref{thm:adfstable2} to the $k$-plane transform.
\subsection{Lower bounds for the $k$-plane transform}\label{intk}
For integers $1\leq k<d$ the \textit{$k$-plane transform} $T_{k,d}$, acting on functions $f:\mathbb{R}^d\rightarrow \mathbb{C}$ is given by
$$
T_{k,d}f(\pi,y)=\int_\pi f(x+y)\mathrm{d}x.
$$
Here $\pi$ belongs to the Grassmannian manifold $\mathcal{G}_{k,d}$ of $k$-dimensional subspaces of $\mathbb{R}^d$, and $y$ belongs to $\pi^\perp$, the orthogonal complement of $\pi$ in $\mathbb{R}^d$. The measure $\mathrm{d}x$ is the $k$-dimensional Hausdorff measure on $\mathbb{R}^d$, allowing us to interpret $T_{k,d}f(\pi,y)$ as an integral of $f$ over the $k$-plane $\pi+\{y\}$. The Grassmannian manifold $\mathcal{M}_{k,d}$ of $k$-planes in $\mathbb{R}^d$ is naturally equipped with the rotation and translation-invariant measure 
$$
\int_{\mathcal{M}_{k,d}}F:=\int_{\mathcal{G}_{k,d}}\left(\int_{\pi^\perp}F(\pi,y)\mathrm{d}y\right)\mathrm{d}\pi,
$$
where $\mathrm{d}\pi$ is the unique $\mathrm{O}(d)$-invariant probability measure on $\mathcal{G}_{k,d}$, 
%(induced by the transitive action of $\mathrm{SO}(d)$ {\color{red} Or should this be $\mathrm{O}(d)$?}), 
so that $$\|T_{k,d}f\|_{L^1(\mathcal{M}_{k,d})}=\|f\|_{L^1(\mathbb{R}^d)}$$ by Fubini's theorem. The question of more general Lebesgue norm control of $T_{k,d}$ of the form \begin{equation}\label{upper}\|T_{k,d}f\|_{L^q(\mathcal{M}_{k,d})}\leq C\|f\|_{L^p(\mathbb{R}^d)}\end{equation}
has received considerable attention, both in terms of the existence of a finite $C$, established for general $k$ by Christ \cite{C:k}, and the identification of its optimal value, along with the nature of the extremising functions $f$; see \cite{C:k, Baern, Drouot, Taryn} and the references there for further context, including the origins of such estimates in the $k=d-1$ setting of the Radon transform. We note that elementary scaling considerations identify
\begin{equation}\label{scalingk}
    \frac{1}{d}\left(1-\frac{1}{q}\right)=\frac{1}{d-k}\left(1-\frac{1}{p}\right)
    \end{equation}
    as a necessary condition for \eqref{upper} to hold with some finite constant $C$.

It was shown recently in \cite[Theorem 8.7]{BT} that $T_{k,d}$ satisfies \textit{lower bounds} in all cases permitted by the necessary scaling condition \eqref{scalingk}, provided one looks in the \textit{quasi-Banach range} $0<p\leq 1$.
\begin{theorem}[\cite{BT}]\label{thm:BT}
    Suppose that $0<p,q\leq 1$ satisfy the scaling condition \eqref{scalingk}.
    Then \begin{equation}\label{ineqk}
\|T_{k,d}f\|_{L^q(\mathcal{M}_{k,d})}\geq \|f\|_{L^p(\mathbb{R}^d)}
    \end{equation}
    for all measurable functions $f:\mathbb{R}^d\rightarrow\mathbb{R}_+$.
\end{theorem}
Although this settles the question of positivity, it is not expected that the constant $1$ in \eqref{ineqk} is optimal when $p<1$; see \cite[Question 10.4]{BT}. One of the main consequences of Theorem \ref{thm:adfstable2} is the confirmation of this.

As it turns out, it is natural to do this in the setting of an inequality that is stronger than \eqref{ineqk}.
While it was not explicitly observed in \cite{BT}, a very minor adaptation of their argument (see the forthcoming Section \ref{Sect:preliminaries}) yields the following, which implies Theorem \ref{thm:BT}, and indeed a range of mixed-norm inequalities, by the concavity of the logarithm and an application of Jensen's inequality.
\begin{theorem}\label{thm:BTlog} Suppose that $0<p,q\leq 1$ satisfy the scaling condition \eqref{scalingk}. Then
\begin{equation}\label{Application of Finnerintro}
\log\|f\|_{L^p(\mathbb{R}^d)}\leq \int_{\mathcal{G}_{k,d}}\log\|T_{k,d}f(\pi,\cdot)\|_{L^q(\pi^\perp)}\mathrm{d}\pi
    \end{equation}
    for all measurable functions $f:\mathbb{R}^d\rightarrow\mathbb{R}_+$.
    \end{theorem}
 Our main result for $T_{k,d}$ is the following.
    \begin{theorem}\label{thm:main}  Let $0<p,q< 1$ satisfy the scaling condition \eqref{scalingk}. Then there exists $D=D(p,k,d)>0$ such that 
\begin{equation}\label{ineqmain}
\log\|f\|_{L^p(\mathbb{R}^d)}\leq -D+\int_{\mathcal{G}_{k,d}}\log\|T_{k,d}f(\pi,\cdot)\|_{L^q(\pi^\perp)}\mathrm{d}\pi
    \end{equation} 
    for all measurable functions $f:\mathbb{R}^d\rightarrow\mathbb{R}_+$.
    \end{theorem}
%The reader will quickly verify that the hypotheses of Theorem \ref{thm:main} are satisfied when $k=d-1$, where the $k$-plane transform becomes the classical Radon transform; a further admissible case is evidently when $d=2k$. We do not expect that the restrictions on $k$ and $d$ in the statement of Theorem \ref{thm:main} are necessary, and indeed we expect the conclusion of Theorem \ref{thm:main} to hold whenever $1\leq k\leq d-1$; see Section \ref{Sect:preliminaries} for more on this.
    \begin{corollary}[cf. Theorem \ref{thm:BT}]
Let $0<p,q< 1$ satisfy the scaling condition \eqref{scalingk}. Then there exists a constant $C=C(p,k,d)>1$ such that 
    \begin{equation}\label{ineqkmain}
\|T_{k,d}f\|_{L^q(\mathcal{M}_{k,d})}\geq C\|f\|_{L^p(\mathbb{R}^d)}
    \end{equation}
    for all measurable functions $f:\mathbb{R}^d\rightarrow\mathbb{R}_+$.
    \end{corollary}
    We conclude with some contextual remarks.
\begin{remark}
Establishing the positivity of the optimal constant $D(p,k,d)$ in \eqref{ineqmain} (or that $C(p,k,d)>1$ in \eqref{ineqkmain}) is, of course, far from identifying its value and providing any information about extremising sequences (Question 10.4 of \cite{BT}). An ostensibly simpler endeavour is to consider the associated local optimisation problem for the functional $$\Lambda(f):=\int_{\mathcal{G}_{k,d}}\log\|T_{k,d}f(\pi,\cdot)\|_{L^q(\pi^\perp)}\mathrm{d}\pi-\log\|f\|_{L^p(\mathbb{R}^d)},$$ where $0<p,q<1$ satisfy \eqref{scalingk}.
Provided $h/f\in L^\infty$, a routine variational argument using Taylor's theorem, reveals that
\begin{eqnarray*}
\begin{aligned}
\frac{\Lambda(f+\varepsilon h)-\Lambda(f)}{\varepsilon}= 
\int_{\mathbb{R}^d}\left(\int_{\mathcal{G}_{k,d}}\frac{\left[T_{k,d}f(\pi, P_{\pi^\perp} x)\right]^{q-1}}{\|T_{k,d}f(\pi,\cdot)\|_q^q}\mathrm{d}\pi-\frac{f(x)^{p-1}}{\|f\|_p^p}\right)h(x)\mathrm{d}x +O(\varepsilon),
\end{aligned}
\end{eqnarray*}
meaning that the expression in parentheses must vanish for almost every $x\in\supp(f)$  for any suitably regular (e.g. continuous) local extremiser $f$ -- should one exist. This allows one to eliminate various candidates for extremisers, such as indicator functions of balls, or Gaussians.
We note that explicit \textit{upper} bounds for $D$, should one require them, are easily obtained by testing on particular functions, such as the Gaussian function $f(x)=e^{-\pi |x|^2}$, which yields 
\begin{equation}\label{inter5} D\leq \frac{d\log p}{2p}-\frac{(d-k)\log q}{2q}.\end{equation}
We refer to Section \ref{Sect:cts} for further discussion on the nature of the optimal constant $D$.
\end{remark}

    \begin{remark}[Entropic interpretations]
As will be expected given Remark \ref{renyisub}, the inequality \eqref{ineqmain} may also be restated in terms of the R\'enyi entropy $\mbox{h}_p$. Specifically, \eqref{ineqmain} becomes
\begin{equation}\label{renyiform}
\mbox{h}_p(f)\leq -\widetilde{D}+\frac{d}{d-k}\int_{\mathcal{G}_{k,d}}\mbox{h}_q(T_{k,d}f(\pi,\cdot))\mathrm{d}\pi,
\end{equation}
where $\widetilde{D}=\frac{p}{1-p}D(p,k,d)$.
The upper bound \eqref{inter5} reveals that $\widetilde{D}\rightarrow 0$ as $p\rightarrow 1$, along with the sharp Shannon entropy inequality 
$$
\mbox{h}(f)\leq \frac{d}{d-k}\int_{\mathcal{G}_{k,d}}\mbox{h}(T_{k,d}f(\pi,\cdot))\mathrm{d}\pi.
$$
This final inequality is also a consequence of Shearer's inequality using an elementary averaging argument; see \cite[Section 8]{BT}.
    \end{remark}
    \begin{remark}[Connections with projection theorems]
      Since $\mbox{h}_p(f)\rightarrow \log |\supp(f)|$ as $p\rightarrow 0$, \eqref{renyiform} reveals that
      \begin{equation}\label{p=0lim}
      \log|A|\leq -E+\frac{d}{d-k}\int_{\mathcal{G}_{k,d}}\log|P_{\pi^\perp}A|\mathrm{d}\pi
      \end{equation}
      for all measurable subsets $A$ of $\mathbb{R}^d$, where $E$ is some nonnegative constant; here $P_{\pi^\perp}$ denotes the orthogonal projection of $\mathbb{R}^d$ onto the subspace $\pi^\perp$.
      As we provide no quantitative lower bounds on $D(p,k,d)$ in Theorem \ref{thm:main}, we may only conclude \eqref{p=0lim} with $E=0$. This may be seen as an 
      averaged form of %the inequality $$\Omega|^{\frac{d-k}{d}}\leq\prod_{j=1}^d|P_{\pi_j^\perp}\Omega|^{\frac 1d},$$ where $\pi_j=\langle e_j, e_{(j+1)\mod d},\hdots, e_{(j+k-1)\mod d}\rangle$, which is an instance 
      the uniform cover inequality of Bollob\'as and Thomason \cite{BoTh} (a case of Finner's inequality in its measure-theoretic form). Interestingly, the $p=0$ limiting inequality \eqref{p=0lim} was shown to hold with the optimal constant $E>0$ recently by Milman and Yehudayoff \cite{MY} in the setting of convex sets $A$; they show that Euclidean balls are extremisers, resulting in an extension of the classical isoperimetric inequality. A notable similarity with our approach is their use of the Brascamp--Lieb--Luttinger rearrangement inequality. We thank Shohei Nakamura for pointing this out.
      %of ...\footnote{Has this been established for all $k$ in the literature? It may be viewed as an averaged form of the uniform cover inequality by Bollob\'as and Thomason, which is an generalization of Loomis--Whitney inequality.}       For convex $\Omega$ this inequality was improved... It is conceivable that a similar limiting argument applied to the stronger \eqref{ineqmain} with the optimal $D(p,k,d)$ recovers this, while extending it to general (non-convex) $\Omega$.
    \end{remark}
%\begin{remark}[Connection to adjoint Brascamp--Lieb constants] Our arguments do not appear to easily yield any quantitative lower bound on the constant $D(p,k,d)$. This constant may be interpreted as a certain adjoint Brascamp--Lieb constant in the setting of the \textit{continuous} Brascamp--Lieb inequalities introduced by Barthe \cite{Barthe cts}. We clarify this in Section \ref{Sect:cts}.\end{remark}

\subsection*{Outline of the paper} 
In Sections \ref{sec:stability} and \ref{sec:stadfi} we prove Theorems \ref{thm:stablefinner} and \ref{thm:adfstable2}. In Section \ref{Sect:preliminaries} we establish Theorem \ref{thm:main} using Theorem \ref{thm:adfstable2} along with an application of the Brascamp--Lieb--Luttinger rearrangement inequality. In Section \ref{sec:gs}, we  partially reprove Theorem \ref{thm:main} using the theory of adjoint Brascamp--Lieb inequalities \cite{BT}, Gaussian stability properties of the Brascamp--Lieb inequality \cite{BC:stable}, and some frame theory. Finally, in Section \ref{Sect:cts}, we interpret Theorem \ref{thm:main} in the setting of adjoint forms of the continuous Brascamp--Lieb inequalities.

%In Section \ref{Sect:preliminaries} we use the theory of adjoint Brascamp--Lieb inequalities \cite{BT} to reduce Theorem \ref{thm:main}, and improved lower bounds on the $k$-plane transform more generally, to a stability problem in the classical theory of Brascamp--Lieb inequalities. This problem is then addressed in Section \ref{Sect:fusion} by combining a recent stability result of Christ and the first author \cite{BC:stable} with a characterisation of tight fusion frames due to Casazza, Fickus, Mixon, Wang and Zhou \cite{CFMWZ}. Finally, we interpret our results in the setting of adjoint forms of the continuous Brascamp--Lieb inequalities of Barthe \cite{Barthe cts} in Section \ref{Sect:cts}.

%In Section \ref{Sect:preliminaries} we use the theory of adjoint Brascamp--Lieb inequalities \cite{BT} to reduce Theorem \ref{thm:main}, and improved lower bounds on the $k$-plane transform more generally, to a stability problem in the classical theory of Brascamp--Lieb inequalities. This problem is then addressed in Section \ref{Sect:fusion} by combining a recent stability result of Christ and the first author \cite{BC:stable} with a characterisation of tight fusion frames due to Casazza, Fickus, Mixon, Wang and Zhou \cite{CFMWZ}. Finally, we interpret our results in the setting of adjoint forms of the continuous Brascamp--Lieb inequalities of Barthe \cite{Barthe cts} in Section \ref{Sect:cts}.

%%%%%%%%%%%%%%%%%%%%%%%%%%%%%%%%%%%%%%%%
\section{Stability for Finner's inequality}\label{sec:stability}
A common approach to establishing stability for a given inequality is to follow a suitable proof of that inequality and examine what occurs at various stages when the inequality is nearly attained. Both Finner's inequality and its adjoint are consequences of multiple applications of H\"older's inequality, and so the following stable form of H\"older's inequality will play an important role in our proofs of both Theorems \ref{thm:stablefinner} and \ref{thm:adfstable2}.
\begin{theorem}[\cite{Aldaz,Aldaz2}]\label{thm:aldaz}
    Suppose $p_j\ge1$ are such that $\sum_{j=1}^mp_j^{-1}=1$, and suppose that $0\not=f_j\in L^{p_j}(\Omega)$ are nonnegative functions for $1\leq j\leq m$. Then 
\[
\int\prod_{j=1}^mf_j\le \prod_{j=1}^m\|f_j\|_{p_j}\Bigg(1-\frac{p_{\max}}{p_{\max}-1}\sum_{j=1}^m\frac 1{p_j}\Big\|\frac{f_j^{p_j/2}}{\|f_j\|_{p_j}^{p_j/2}}-\sum_{\ell=1}^m\frac{1}{p_\ell}\frac{f_\ell^{p_\ell/2}}{\|f_\ell\|_{p_\ell}^{p_\ell/2}}\Big\|_2^2\Bigg),
\]
where $p_{\max}=\max\{p_1,\hdots,p_m\}$.
\end{theorem}
Consequently, if H\"older's inequality is nearly saturated in the sense that
\[
\int\prod_{j=1}^mf_j\ge (1-\delta) \prod_{j=1}^m\|f_j\|_{p_j}
\]
for some $\delta>0$,
then the normalised (nonnegative) functions $f_j^{p_j/2}$ are necessarily close to one another in $L^2$ in the sense that
\[
\Big\|\frac{f_j^{p_j/2}}{\|f_j\|_{p_j}^{p_j/2}}-\frac{f_\ell^{p_\ell/2}}{\|f_\ell\|_{p_\ell}^{p_\ell/2}}\Big\|_2\lesssim \delta^{\frac 12}
\]
for all $j,\ell$. We shall repeatedly exploit this stability property in this section.

Before we begin our proof of Theorem \ref{thm:stablefinner}
we briefly review the classical proof \cite{finner} of Finner's inequality,
which proceeds by induction on $d$ and H\"older's inequality. Setting $I_{d-1}=\{1,\hdots,d-1\}$ and $M_d=\{j:d\in S_j\}$, we write
\begin{equation}
    \label{ineq:pfoffinner}
\begin{aligned}
    \int_{\Omega}\prod_{j=1}^mg_j\circ P_{j}\mathrm{d}\mu&=\int_{\Omega}\prod_{j\in M_d}g_j\circ P_{j}\prod_{j\notin M_d}g_j\circ P_{j}\mathrm{d}\mu\\
    &\le \int_{\Omega_{I_{d-1}}}\prod_{j\in M_d}\Big(\int_{\Omega_d} (g_j\circ P_{j})^{q_j}\mathrm{d}\mu_d\Big)^{\frac 1{q_j}}\cdot\prod_{j\notin M_d}g_j\circ P_{j}\mathrm{d}\mu_{I_{d-1}}\\
    &\le \prod_{j=1}^m\|g_j\|_{q_j}.
\end{aligned}
\end{equation}
In the second line we have used H\"older's inequality with the hypothesis $\sum_{j\in M_d}q_j^{-1}=1$, and for the last line, we have used Finner's inequality on $\Omega_{I_{d-1}}=\prod_{i=1}^{d-1}\Omega_i$. With this argument in mind we turn to the proof of Theorem \ref{thm:stablefinner}.
\subsection{Proof of Theorem \ref{thm:stablefinner}}
 When $d=1$, Theorem \ref{thm:aldaz} directly implies Theorem \ref{thm:stablefinner}. Indeed, in this case, $P_j$ is the identity map on $\Omega$ for all $j$. By Theorem \ref{thm:aldaz},  \eqref{finnerstablecond} tells us that
   \[
   \Big\|\frac{g_j^{q_j/2}}{\|g_j\|_{q_j}^{q_j/2}}-\frac{g_\ell^{q_\ell/2}}{\|g_\ell\|_{q_\ell}^{q_\ell/2}}\Big\|_2\lesssim \delta^{\frac 12}, 
   \]
   for all $j,\ell\in\{1,\hdots,m\}$.
   Taking $h_1=g_1^{q_1}$ and $c_j=\|g_j\|_{q_j}^{q_j}/\|g_1\|_{q_1}^{q_1}$, the above estimate implies the desired inequality \eqref{finnerstablest}. Evidently we reach the same conclusion if we take $h_1=g_\ell^{q_\ell}$ for any $\ell=1,\hdots,m$.
   
   Turning to the inductive step, we suppose that Theorem \ref{thm:stablefinner} holds for $d-1$. Define 
    \[
    M_0:=\{j:S_j=\{1,\hdots,d\}\}.
    \]
    We first treat the case $M_0=\emptyset$, which means that all of the $P_j$ are projections onto proper subspaces of $\Omega$. This case includes the Loomis--Whitney inequality of course.
    %The proof follows the arguments used in the proof of Theorem \ref{thm:stableLW}.

Assume that $M_0=\emptyset$ and fix $j(i)\in M_i$ for all $i=1,\hdots, d$. By \eqref{ineq:pfoffinner}, the condition \eqref{finnerstablecond} implies that 
\[
\int_{\Omega_{I_{d-1}}}\prod_{j\in M_d}\Big(\int (g_j\circ P_{j})^{q_j}\mathrm{d}\mu_d\Big)^{\frac 1{q_j}}\cdot\prod_{j\notin M_d}g_j\circ P_{j}\mathrm{d}\mu_{I_{d-1}}\ge(1-\delta) \prod_{j=1}^m\|g_j\|_{q_j}.
\]
Let $\tilde S_j:=S_j\setminus \{d\}$ and $\tilde P_j:\Omega_{I_{d-1}}\rightarrow \Omega_{\tilde S_j}$ denote the corresponding projection maps.
Define $\tilde g_j$ by $\tilde g_j=g_j$ for $j\notin M_d$, and 
\[
\tilde g_j(\tilde P_j\tilde x):=\Big(\int g_j(P_{j}(\tilde x,x_d))^{q_j}\mathrm{d}\mu_d(x_d)\Big)^{\frac 1{q_j}}
\]
for $j\in M_d$ and $\tilde x=(x_i)_{i\in I_{d-1}}\in \Omega_{I_{d-1}}$. It follows that
\[
\int_{\Omega_{I_{d-1}}}\prod_{j=1}^m\tilde g_j\circ \tilde P_{j}\mathrm{d}\mu_{I_{d-1}}\ge(1-\delta) \prod_{j=1}^m\|\tilde g_j\|_{q_j}
\]
and 
$$\sum_{j=1}^m\frac{1}{q_j}\mathrm{1}_{\tilde{S}_j}(i)=1$$ for all $1\leq i\leq d-1$.
Thus, by the inductive hypothesis,
\eqref{finnerstablecond} implies that
\[
\|\tilde g_j^{\frac{q_j}{2}}-\tilde H_{j}^{\frac 12}\|_2\lesssim \delta^{\frac 12} \|\tilde g_j\|_{q_j}^{\frac{q_j}2 }
\]
for all $j$, where $\tilde H_{j}(\tilde P_j\tilde x)= c_{j}\prod_{i\in \tilde S_j}h_{i}(x_i)$ and  
\[
\begin{aligned}
    h_i(x_i)&=\int (\tilde g_{j(i)}\circ \tilde P_{j(i)})^{q_{j(i)}}\mathrm{d}\mu_{\tilde S_{j(i)}\setminus \{i\}}\\
    &=\int (g_{j(i)}\circ  P_{j(i)} )^{q_{j(i)}}\mathrm{d}\mu_{ S_{j(i)}\setminus \{i\}}
\end{aligned}
\]
for each $i=1, \hdots, d-1$.
Since $\tilde g_j=g_j$ and $\tilde H_j(\tilde P_j\tilde x)=H_j(P_jx)$ for $j\notin M_d$, we have the desired inequality \eqref{finnerstablest} for such $g_j$.
By the same argument replacing $M_d$ by $M_i$, \eqref{finnerstablecond} also gives that
\[
\|g_j^{\frac{q_j}{2}}-H_{j}^{\frac 12}\|_2\lesssim \delta^{\frac 12} \|g_j\|_{q_j}^{\frac{q_j}2 },
\]
for $j\notin M_i$. Since $M_0=\emptyset$, for each $j$ there exists $i$ such that $j\notin M_i$, which gives the desired result.

Now assume that $M_0\neq \emptyset$, meaning that some of the projections $P_j$ are the identity map on $\Omega$. Set 
\begin{equation}\label{for Holder}
\frac{1}{q}:=1-\sum_{j\in M_0}\frac 1{q_j},
\end{equation}
and note that $q<\infty$ by \eqref{assumptionM}.
By H\"older's inequality,
\[
\begin{aligned}
     \int_{\Omega}\prod_{j=1}^mg_j\circ P_{j}\mathrm{d}\mu&=\int_{\Omega}\prod_{j\in M_0}g_j\circ P_{j}\prod_{j\notin M_0}g_j\circ P_{j}\mathrm{d}\mu\\
     &\le \Big(\prod_{j\in M_0}\|g_j\|_{q_j}\Big)\Big(\int \prod_{j\notin M_0}(g_j\circ P_{j})^q\mathrm{d}\mu\Big)^{\frac 1q}\\
     &\le \prod_{j=1}^m\|g_j\|_{q_j}.
\end{aligned}
\]
In the last inequality we have applied Finner's inequality with the condition \begin{equation}\label{inter4}
\sum_{j\notin M_0}\frac q{q_j}1_{S_j}\equiv 1
\end{equation}
coming from \eqref{for Holder} and the assumption \eqref{assumptionM}. Thus, if \eqref{finnerstablecond} holds, then
\begin{equation}\label{inter1}
\int_{\Omega}\prod_{j\in M_0}g_j\circ P_{j}\prod_{j\notin M_0}g_j\circ P_{j}\mathrm{d}\mu \ge(1-\delta) \Big(\prod_{j\in M_0}\|g_j\|_{q_j}\Big)\Big(\int \prod_{j\notin M_0}(g_j\circ P_{j})^q\mathrm{d}\mu\Big)^{\frac 1q}
\end{equation}
and
\begin{equation}\label{inter2}
\int \prod_{j\notin M_0}(g_j\circ P_{j})^q\mathrm{d}\mu \ge (1-\delta)^q \prod_{j\notin M_0}\|g_j^q\|_{\frac{q_j}q}.
\end{equation}
Note that $q_j/q\ge1$ for $j\notin M_0$.
Using \eqref{inter2} we may apply the argument used above (in the case $M_0=\emptyset$) to deduce \eqref{finnerstablest} from \eqref{finnerstablecond} for all $j\notin M_0$. Further, applying Theorem \ref{thm:aldaz} to \eqref{inter1},
\[
\Big\|g_j^{\frac{q_j}{2}}-\tilde c_j\prod_{\ell\notin M_0}(g_\ell\circ P_\ell)^{\frac q2}\Big\|_2\lesssim  \delta^{\frac 12} \|g_j\|_{q_j}^{\frac{q_j}{2}}
\]
for all $j\in M_0$.
Additionally, we observe that $\|g_\ell\|_{q_\ell}^{q_\ell/2}\sim \|H_\ell^{1/2}\|_2=c_\ell^{1/2}\prod_{i\in S_\ell}\|h_i\|_1^{1/2}$ for $\ell\notin M_0$ and
\[
\|g_j\|_{q_j}^{\frac {q_j}2}\sim \tilde c_j\Big\|\prod_{\ell\notin M_0}(g_\ell\circ P_\ell)^{\frac q2}\Big\|_2\sim \tilde c_j\prod_{\ell\notin M_0}\|g_\ell\|_{q_\ell}^{\frac q2},
\]
for $j\in M_0$. To obtain \eqref{finnerstablest} for $j\in M_0$, it therefore suffices to show that 
\begin{equation}
    \label{Fstable2}
\begin{aligned}
\Big\|\prod_{\ell\notin M_0}(g_\ell\circ P_\ell)^{\frac q2}-\prod_{\ell\notin M_0}(H_\ell\circ P_\ell)^{\frac {q}{2q_\ell}}\Big\|_2&\lesssim  \delta^{\frac 12}\prod_{\ell\notin M_0}\|g_{\ell}\|_{q_{\ell}}^{\frac q2}.
\end{aligned}
\end{equation}
Indeed, combining the above three inequalities, for $j\in M_0$, we obtain that $$\Big\|g_j^{\frac {q_j}2}-\tilde c_j\prod_{\ell\notin M_0}(H_\ell\circ P_\ell)^{ \frac {q}{2q_\ell}}\Big\|_2\lesssim  \delta^{\frac 12}\|g_j\|_{q_j}^{\frac {q_j}2}.$$  Thus, \eqref{finnerstablest} for $j\in M_0$ follows by taking $c_j=\tilde c_j^{2}\prod_{\ell\notin M_0}c_\ell^{q/q_\ell}$, since
\begin{equation}
    \label{Hj=cj'hi}
\tilde c_j^2\prod_{\ell\notin M_0}(H_\ell\circ P_\ell)^{ \frac {q}{q_\ell}}=c_j \prod_{i=1}^d h_i=H_j,
\end{equation}
by \eqref{inter4}.

It remains to prove \eqref{Fstable2}. 
We first deal with the portion of the domain in which $g_\ell^{q_\ell}(P_\ell x)$ is significantly larger than $H_\ell(P_\ell x)$ for sufficiently many indices $\ell$. Defining sets $ \tilde A_{\ell}\subset \Omega_{S_\ell}$ by 
\[
 \tilde A_{\ell}=\left\{y\in\Omega_{S_\ell}: g_\ell(y)^{ {q_\ell}}\ge 10^2H_\ell(y)\right\},
\]
it immediately follows that
\begin{equation}
    \label{ineq:gellHellerror}
\|g_\ell^{\frac {q_\ell}2}1_{\tilde A_{\ell}}\|_2+\|H_\ell^{\frac 12}1_{\tilde A_{\ell}}\|_2\lesssim \|(g_\ell^{\frac {q_\ell}2}-H_\ell^{\frac 12})1_{\tilde A_\ell}\|_2\lesssim  \delta^{\frac 12} \|g_\ell\|_{q_\ell}^{\frac{q_\ell}2}.
\end{equation}
For a subset $L\subset M_0^c$, define a set $A_L$ in $\Omega$ by
\[
A_L:=\bigcap_{\ell\in L}P_\ell^{-1}(\tilde A_\ell),
\]
so that $g_\ell^{q_\ell}(P_\ell x)\geq 10^2H_\ell(P_\ell x)$ for all $\ell\in L$ and $x\in A_L$.
We distinguish two cases according to whether $L$ satisfies 
\begin{equation}
    \label{conditiononL}
\sum_{\ell\in L}\frac q{q_\ell}\ge1.
\end{equation}
%This condition will ensure that $L$ contains sufficiently many indices $\ell$ for which $g_\ell^{q_\ell/2}(P_\ell x)\geq 10H_\ell(P_\ell x)$ for $x\in A_L$.

We begin by showing that
\begin{equation}
\label{ineq:differenceonA}
\begin{aligned}
    \Big\|\prod_{\ell\notin M_0}(g_\ell\circ P_\ell)^{\frac q2}-\prod_{\ell\notin M_0}(H_\ell\circ P_\ell)^{\frac {q}{2q_\ell}}\Big\|_{L^2(  A)}\lesssim \delta^{\frac 12} \prod_{\ell\notin M_0}\|g_\ell\|_{q_\ell}^{\frac q2},
\end{aligned}    
\end{equation}
where  
\[
A:=\bigcup_{L:\sum_{\ell\in L}(q/q_\ell)\ge1} A_{L}.
\]
To prove \eqref{ineq:differenceonA}, it suffices to show that for $L$ satisfying \eqref{conditiononL},
\[
    \Big\|\prod_{\ell\notin M_0}(g_\ell\circ P_\ell)^{\frac q2}-\prod_{\ell\notin M_0}(H_\ell\circ P_\ell)^{\frac {q}{2q_\ell}}\Big\|_{L^2(A_L)}
    \lesssim \delta^{\frac{1}{2}}\prod_{\ell\notin M_0}\|g_\ell\|_{q_\ell}^{\frac q2},
\]
which by the triangle inequality may be reduced to the two estimates \begin{equation}\label{inter3}    \Big\|\prod_{\ell\notin M_0}(g_\ell\circ P_\ell)^{ \frac q2}\Big\|_{L^2(  A_L)},\; \Big\|\prod_{\ell\notin M_0}(H_\ell\circ P_\ell)^{\frac {q}{2q_\ell}}\Big\|_{L^2(  A_L)}\lesssim \delta^{\frac{1}{2}}\prod_{\ell\notin M_0}\|g_\ell\|_{q_\ell}^{\frac q2}.\end{equation} 
Since $P_\ell(A_L)\subset \tilde A_\ell$ for all $\ell\in L$, and recalling \eqref{inter4}, Finner's inequality together with \eqref{ineq:gellHellerror} and \eqref{conditiononL} implies that
\[
\begin{aligned}
    \Big\|\prod_{\ell\notin M_0}(g_\ell\circ P_\ell)^{\frac q2}\Big\|_{L^2(  A_L)}&\le \prod_{\ell\notin M_0}\|g_\ell1_{P_\ell(A_L)}\|_{q_\ell}^{\frac q2}\\
    &\le \prod_{\ell\in L}\|g_\ell1_{\tilde A_\ell}\|_{q_\ell}^{\frac q2}\prod_{\ell'\notin M_0\cup L}\|g_{\ell'}\|_{q_{\ell'}}^{\frac q2} \\
    &\lesssim  \delta^{\sum_{\ell\in L}\frac{q}{2q_\ell}}\prod_{\ell\notin M_0}\|g_\ell\|_{q_\ell}^{\frac q2} \\
    &\le \delta^{\frac 12}\prod_{\ell\notin M_0}\|g_\ell\|_{q_\ell}^{\frac q2}.
\end{aligned}
\]
Similar reasoning also establishes the second estimate in \eqref{inter3}.

Returning to \eqref{Fstable2}, having established \eqref{ineq:differenceonA} we now consider the portion coming from $ A^c$, the complement of $A$. Unlike for the estimate on $A$, we need to confront the difference of the two products. Although \eqref{finnerstablest} holds for all $\ell \notin M_0$, it is not immediately clear how to exploit these inequalities simultaneously. To overcome this, we begin by establishing the pointwise estimate
\begin{equation}
    \label{ineq:AMGMineq}
\Big|\prod_{\ell\notin M_0}\Big(\frac{g_\ell(P_\ell x)^{q_\ell/2}}{H_\ell(P_\ell x)^{1/2}}\Big)^{\frac q {q_\ell}}-1\Big|\lesssim \sum_{\ell\notin M_0}\Big|\frac{g_\ell(P_\ell x)^{q_\ell/2}}{H_\ell(P_\ell x)^{1/2}}-1\Big|,
\end{equation}
allowing us to control the product by a sum whenever $x\in A^c$. This reduces the problem to applying \eqref{finnerstablest} for only a single index $\ell$. It is worth noting that $H_\ell(P_\ell x)>0$ for all $\ell$ and $x\in A^c$. Indeed, suppose that $H_{\ell'}(P_{\ell'} x)=0$ for some $\ell'$ and $x\in A^c$. Then $h_i(x_i)=0$ for some $i\in S_{\ell'}$, which implies $H_\ell(P_\ell x)=0$ for all $\ell\in M_i\cap M_0^c$. By the definition of $A_L$, we have $x\in A_{M_i\cap M_0^c}$. Since $\sum_{\ell\in M_i\cap M_0^c}(q/q_\ell)=1$, it follows that $x\in  A$, a contradiction.
 
 For notational simplicity set
\[
a_\ell=a_\ell(x):=\frac{g_\ell(P_\ell x)^{q_\ell/2}}{H_\ell(P_\ell x)^{1/2}}.
\]
 To establish \eqref{ineq:AMGMineq},  we fix an $x\in A^c$ and define the index set
\[
L_x:=\left\{\ell\in M_0^c : x\in P_\ell^{-1}(\tilde A_\ell) \right\}.
\]
Since $x\in A_{L_x}$ by definition, 
\begin{equation}
    \label{conditiononLx}
\sum_{\ell\in L_x}\frac q{q_\ell}<1.
\end{equation}
Then $x\in P_\ell^{-1}(\tilde A_\ell)$ if and only if  $\ell \in  L_x$, which means that 
\begin{equation}
    \label{ineq:glhlfinite}
    \begin{aligned}
        0\le a_\ell \le 10, \quad &\text{ for }  \ell\in M_0^c\setminus L_x, \\
        a_\ell\ge10,\quad  &\text{ for }  \ell \in L_x.
    \end{aligned}
\end{equation}
By \eqref{conditiononLx} and \eqref{ineq:glhlfinite} it follows that
\begin{equation}
    \label{ineq:AMGMLx}
\Big|\prod_{\ell \in L_x}a_\ell^{\frac q {q_\ell}}-1\Big|\le\prod_{\ell \in L_x}a_\ell^{\frac q {q_\ell}}\le \max_{\ell\in L_x}a_\ell\lesssim  \sum_{\ell\in L_x}|a_\ell-1|.
\end{equation}
%Here, we apply $\prod_{\ell\in L_x}a_\ell^{q/q_\ell}\le\sum_{\ell\in L_x}(q/q_\ell)a_\ell+1-\sum_{\ell\in L_x}(q/q_\ell)$. 
On the other hand, by the mean value theorem,
\begin{equation}
    \label{ineq:AMGMLxc}
    \Big|\prod_{\ell \in M_0^c \setminus L_x}a_\ell^{\frac q {q_\ell}}-1\Big|\lesssim \sum_{\ell\in M_0^c \setminus L_x}|a_\ell-1|.
\end{equation}
Indeed, if $1/2\le a_\ell \le 10 $ for all $\ell \in M_0^c \setminus L_x$, \eqref{ineq:AMGMLxc} follows directly from the mean value theorem. Otherwise, if $a_\ell\le 1/2$ for some $\ell$, then \eqref{ineq:AMGMLxc} is trivial since the right hand side of \eqref{ineq:AMGMLxc} is bounded below by $1/2$ while its left hand side is bounded above by a constant depending only on $q,q_\ell$.
Consequently, using the following inequality from \eqref{ineq:glhlfinite}, 
\[
\begin{aligned}
\Big|\prod_{\ell\notin M_0}a_\ell^{\frac q {q_\ell}}-1\Big|\le& \Big|\prod_{\ell\notin M_0} a_\ell^{\frac q {q_\ell}}-\prod_{\ell\in M_0^c\setminus L_x}a_\ell^{\frac q {q_\ell}}\Big|+\Big|\prod_{\ell\in M_0^c\setminus L_x}a_\ell^{\frac q {q_\ell}}-1\Big|\\
%\le& \prod_{\ell\in M_0^c\setminus L_x}\Big(\frac{g_\ell(P_\ell x)^{q_\ell/2}}{H_\ell(P_\ell x)}\Big)^{\frac q {q_\ell}}\times\Big|\prod_{\ell\in L_x}\Big(\frac{g_\ell(P_\ell x)^{q_\ell/2}}{H_\ell(P_\ell x)}\Big)^{\frac q {q_\ell}}-1\Big|\\
%&+\Big|\prod_{\ell\in M_0^c \setminus L_x}\Big(\frac{g_\ell(P_\ell x)^{q_\ell/2}}{H_\ell(P_\ell x)}\Big)^{\frac q {q_\ell}}-1\Big|\\
\lesssim&\Big|\prod_{\ell\in L_x}a_\ell^{\frac q {q_\ell}}-1\Big|+\Big|\prod_{\ell\in M_0^c \setminus L_x}a_\ell^{\frac q {q_\ell}}-1\Big|,    
\end{aligned}
\]
 together with \eqref{ineq:AMGMLx} and \eqref{ineq:AMGMLxc}, we conclude that \eqref{ineq:AMGMineq} holds for $x\in A^c$.
We note that \eqref{ineq:AMGMineq} does not depend on the index set $L_x$.

We now prove an analogue of \eqref{ineq:differenceonA} on $A^c$, using \eqref{ineq:AMGMineq}. The pointwise estimate \eqref{ineq:AMGMineq} yields that
\begin{equation}
\label{ineq:Ei/A}
\begin{aligned}
        \Big\|\prod_{\ell\notin M_0}(g_\ell\circ P_\ell)^{\frac q2}&-\prod_{\ell\notin M_0}(H_\ell\circ P_\ell)^{\frac {q}{2q_\ell}}\Big\|_{L^2(A^c )}\\
        &=\Big\|\Big(\prod_{\ell\notin M_0}\Big(\frac{(g_\ell\circ P_\ell)^{q_\ell/2}}{(H_\ell\circ P_\ell)^{1/2}}\Big)^{\frac q {q_\ell}}-1\Big)\prod_{\ell'\notin M_0}(H_{\ell'}\circ P_{\ell'})^{\frac {q}{2q_{\ell'}}}\Big\|_{L^2(A^c )} \\
        &\lesssim \sum_{\ell\notin M_0}\Big\|\Big(\frac{(g_\ell\circ P_\ell)^{q_\ell/2}}{(H_\ell\circ P_\ell)^{1/2}}-1\Big)\prod_{\ell'\notin M_0}(H_{\ell'}\circ P_{\ell'})^{\frac {q}{2q_{\ell'}}}\Big\|_2.
\end{aligned}    
\end{equation}
Since $$(H_\ell\circ P_\ell)^{-1}\prod_{{\ell'}\notin M_0}(H_{\ell'}\circ P_{\ell'})^{ {q}/{q_{\ell'}}}=c_\ell^{-1}\prod_{\ell'\notin M_0}c_{\ell'}^{q/q_{\ell'}}\prod_{i\notin S_\ell}h_i$$
from \eqref{Hj=cj'hi}, Fubini's theorem tells us that
\[
\begin{aligned}
    \Big\|\Big(\frac{(g_\ell\circ P_\ell)^{q_\ell/2}}{(H_\ell\circ P_\ell)^{1/2}}-1\Big)&\prod_{\ell'\notin M_0}(H_{\ell'}\circ P_{\ell'})^{\frac {q}{2q_{\ell'}}}\Big\|_2^2\\
    &=c_\ell^{-1}\prod_{\ell'\notin M_0}c_{\ell'}^{\frac q{q_{\ell'}}}\int\Big(\int \big(({g_\ell\circ P_\ell)^{\frac {q_\ell}2}}-{(H_\ell\circ P_\ell )^{\frac 12}}\big)^2\mathrm d\mu_{S_{\ell}}\Big)\prod_{i\notin S_\ell}h_i\mathrm d\mu_{S_\ell^c}\\
    &=c_\ell^{-1}\prod_{\ell'\notin M_0}c_{\ell'}^{\frac q{q_{\ell'}}}\|g_\ell^{\frac {q_\ell}2}-H_\ell^{\frac 12}\|_2^2\prod_{i\notin S_\ell}\|h_i\|_1.
\end{aligned}
\]
By this, \eqref{ineq:Ei/A}, and the observation that $\|g_\ell^{q_\ell/2}-H_\ell^{1/2}\|_2\lesssim c_\ell^{1/2}\delta^{1/2}\prod_{i\in S_\ell}\|h_i\|_1^{1/2}$ from \eqref{finnerstablest} for $\ell\notin M_0$, we conclude that
\[
\begin{aligned}
\|\prod_{\ell\notin M_0}(g_\ell\circ P_\ell)^{\frac q2}-\prod_{\ell\notin M_0}(H_\ell\circ P_\ell)^{\frac {q}{2q_\ell}}\|_{L^2( A^c)}&\lesssim \delta^{\frac 12} \prod_{\ell\notin M_0}c_{\ell}^{\frac q{2q_{\ell}}}\prod_{i=1}^d\|h_i\|_1^{\frac 12}\\
    &= \delta^{\frac 12} \prod_{\ell\notin M_0}\Big(c_\ell\prod_{i\in S_\ell}\|h_i\|_1\Big)^{\frac q{2q_\ell}}\\
    &\lesssim \delta^{\frac 12} \prod_{\ell\notin M_0}\|g_\ell\|_{q_\ell}^{\frac q2}.
\end{aligned}
\]
Finally, together with \eqref{ineq:differenceonA}, this establishes the desired inequality \eqref{Fstable2}.

\subsection{Sharpness of the exponent in \eqref{finnerstablest}}\label{sec:sharpness}
In this subsection, we show that the order of $\delta$ in \eqref{finnerstablest} is optimal, except for the Fubini case.  By the Fubini case, we mean the case in which $M_i$ consists of a single element for every $i$.  For this case,
\[
\|g_j^{\frac {q_j}2}-H_j^{\frac 12}\|_2=0
\]
for all $j$, so the optimal exponent may be regarded as infinity.

We now assume that  $M_i$ contains at least two elements for some $i$. We further assume that each $\Omega_i$ contains two disjoint measurable sets $A_i,B_i$ satisfying $0<\mu_i(A_i),\mu_i(B_i)<\infty$. This is a mild nondegeneracy assumption since if every measurable subset of $\Omega_i$ has either zero measure or full measure, then $L^{q}(\Omega_i)$ is at most one-dimensional, allowing us to eliminate the corresponding factor $\Omega_i$ from the Finner data.
%that, in particular, excludes the case in which some $\Omega_i$ consists of a single element. 
Under these assumptions, we obtain the following proposition, which shows that the exponent $1/2$ is optimal.
\begin{proposition}\label{prop:sharpfinner}
    Let $\Omega_i$ and $M_i$ satisfy the above assumptions. Suppose that \eqref{finnerrelation} and \eqref{assumptionM} hold. Then there exist nonnegative functions $g_1,\hdots, g_m$ satisfying  \eqref{finnerstablecond} such that $$  \sum_{j=1}^m\frac{\|g_j^{{q_j}/{2}}-H_j^{ 1/2}\|_{L^2(\Omega_{S_j})}}{\|g_j\|_{L^{q_j}(\Omega_{S_j})}^{{q_j}/{2}}}\gtrsim {\delta^{\frac 12}} ,$$ for all $c_j>0$ and nonnegative functions $h_i$, where $H_j=c_j\prod_{i\in S_j}h_i$.
\end{proposition}
\begin{proof}
Without loss of generality, assume that $M_d$ contains at least two elements. 
   Choose $j_0,j_1\in M_d$ satisfying $j_0\neq j_1$. By our assumptions on $\Omega_i$, there exist $A_i,B_i\subset \Omega_i$ such that $A_i\cap B_i=\emptyset$ and $0<\mu_i(A_i),\mu(B_i)<\infty$. For simplicity set $E_i:=A_i\cup B_i$ and after replacing $\mu_i$ by $\mu_i/\mu_i(E_i)$, assume that $\mu_i(E_i)=1$.
   Then define
\[
   \begin{aligned}
    g_{j}(P_jx)=\prod_{i\in S_j}1_{E_i}(x_{i}),&  \quad \text{ for }j\neq j_0, j_1, \\
    g_{j_n}(P_{j_n}x)=\big(1_{E_d}(x_d)+(-1)^nh(x_d)&\big)\prod_{i\in S_{j}\setminus\{d\}}1_{E_i}(x_{i}), \quad \text{ for }n=0,1,
   \end{aligned}
   \]
where $$h(x_d)=\delta^{\frac 12}\mu_d(B_d)1_{A_d}(x_d)-\delta^{\frac 12}\mu_d(A_d)1_{B_d}(x_d).$$
We note that $\mu_d(A_d)+\mu_d(B_d)=1$.

First, one can observe that $\|g_j\|_{q_j}=1,$ for $j\neq j_0,j_1$. When $j=j_0,j_1$, using the following inequality for sufficiently small $\delta>0$,
\[
\begin{aligned}
\|1_{E_d}+(-1)^nh\|_{q}^q&=(1+(-1)^n\mu_d(B_d)\delta^{\frac 12})^q\mu_d(A_d)+(1+(-1)^{n+1}\mu_d(A_d)\delta^{\frac 12})^q\mu_d(B_d)    \\
&=1+O(\delta),
\end{aligned}
\]
we have $\|g_{j_n}\|_{q_{j_n}}=1+O(\delta)$ for $n=0,1$. On the other hand, the left hand side of Finner's inequality for this data satisfies
\[
\begin{aligned}
    \int \prod_{j=1}^mg_j(P_jx)\mathrm d\mu=&\int_{E_1\times\cdots\times E_{d-1}\times A_d}\left(1-\delta \mu_{d}(B_d)^2\right) \mathrm d\mu\\
    &+\int_{E_1\times\cdots\times E_{d-1}\times B_d}\left(1-\delta\mu_{d}(A_d)^2\right) \mathrm d\mu\\
    \ge& 1-\delta,
\end{aligned}
\]
 which implies \eqref{finnerstablecond} with  $\delta$ replaced by $O(\delta)$. 
Thus, it suffices to show that 
\begin{equation}
    \label{ineq:sharpness}
\sum_{j=1}^m{\|g_{j}^{\frac{q_{j}}{2}}-H_{j}^{ \frac 12}\|_{L^2(\Omega_{S_{j}})}}\ge \frac  {c\delta^{\frac 12}}4,
\end{equation}
for all $h_i$, where $c=\delta^{-1/2}\|h\|_2$. We note that $c$ is independent of $\delta$.

We prove \eqref{ineq:sharpness} by contradiction. Assume that there exist nonnegative functions $h_i$ such that
\[
 \|g_{j}^{\frac{q_j}{2}}-H_{j}^{\frac 12}\|_2< \frac  {c\delta^{\frac 12}}4
\]
for all $j$. For $n=0,1$ we observe that
%\footnote{JB: an inequality below? Also some square roots missing on the $h_j$.;Thanks!} 
\[
\|g_{j_n}^{\frac{q_{j_n}}{2}}-H_{j_n}^{\frac 12}\|_2\ge\Big(\int \prod_{i\in S_{j_n}\setminus \{d\}}1_{E_i}\int\Big|1_{E_d}+(-1)^nh-c_{j_n}^{\frac 12}\prod_{i\in S_{j_n}}h_i^{\frac 12}\Big|^2\mathrm d\mu_d\mathrm d\mu_{S_{j_n}\setminus\{d\}}\Big)^{\frac 12}.
\]
 By the assumption, it follows that there exist constants $\tilde c_{j_n}\ge0$ such that
\[
\begin{aligned}
\|1_{E_d}+(-1)^nh-\tilde c_{j_n}h_d^{\frac 12}\|_2< \frac {c\delta^{\frac 12}}4,
\end{aligned}
\]
for $n=0,1$.
Consequently, we have $$1+O(\delta)-\frac c4\delta^{\frac 12}< \tilde c_{j_n}\|h_d\|_1^{ 1/2}< 1+O(\delta)+\frac c4\delta^{\frac 12}.$$ Combining these inequalities, we observe that
\[
\|1_{E_d}+h-(1_{E_d}-h)\|_2< \frac c2\delta^{\frac 12}+|\tilde c_{j_1}-\tilde c_{j_0}|\|h_d\|_1^{\frac 12}< 2c\delta^{\frac 12},
\]
which is a contradiction by the definition of $c$. Thus,   \eqref{ineq:sharpness} holds for all nonnegative functions $h_i$.
\end{proof}

\begin{remark}\label{rmk:sharpnessCD}
Theorem \ref{thm:stablefinner} sharpens the stability result of Cook and Dembo \cite{CD}. Indeed, by H\"older's inequality, $$\|g_j^{q_j}-H_j\|_1\le\|g_j^{q_j/2}-H_j^{1/2}\|_2\|g_j^{q_j/2}+H_j^{1/2}\|_2.$$ Thus, the $L^2$ estimate \eqref{finnerstablest} yields that 
\begin{equation}\label{cd1}
 \|g_j^{{q_j}}-H_j\|_{L^1(\Omega_{S_j})}\leq C\delta^{\frac 12} \|g_j\|_{L^{q_j}(\Omega_{S_j})}^{{q_j}}.
\end{equation}
The exponent $1/2$ is optimal except for the Fubini case, as discussed in \cite{CD}. Indeed, an $L^1$ analogue of the argument establishing Proposition \ref{prop:sharpfinner} reveals the sharpness of the exponent $1/2$ in \eqref{cd1}. It is worth noting that characteristic functions may be used to prove Proposition \ref{prop:sharpfinner}, whereas they do not suffice to verify sharpness in the  $L^1$ setting; see Remark \ref{Rem:EFKY}.
\end{remark}

\begin{comment}
   
   \[
   \begin{aligned}
   g_{j}(P_jx)=\prod_{i\in S_j}1_{[0,1]}(x_{i}),&  \quad \text{ for }j\neq j_1, \\
    g_{j}(P_{j}x)=1_{[\delta,1+\delta]}(x_d)\prod_{i\in S_{j}\setminus\{d\}}1_{[0,1]}(x_{i}),& \quad \text{ for }j= j_1,
   \end{aligned}
   \]
Then $\|g_j\|_{q_j}=1$ for all $j$ and
\[
\int \prod_{j=1}^mg_j(P_jx)\mathrm dx=\int 1_{[\delta,1]}(x_d)\prod_{i\in I_{d-1}}1_{[0,1]}(x_i)\mathrm dx=1-\delta,\]
which gives \eqref{finnerstablecond}. 

A modification of this example also shows Proposition \ref{prop:sharpfinner} holds for general $\Omega$, provided that each $\Omega_i$ contains two disjoint measurable sets $A_i,B_i$ satisfying $0<\mu_i(A_i)=\mu_i(B_i)<\infty$. This is a mild nondegeneracy assumption that, in particular, excludes the case in which some $\Omega_i$ consists of a single element.
    
\end{comment}

\subsection{Stability in the measure-theoretic setting}\label{Sect:stabmeas}
Theorem \ref{thm:stablefinner} implies the following stability result in the measure-theoretic setting. In the Euclidean case this recovers a result of Ellis, Friedgut, Kindler and Yehudayoff \cite{EFKY}; see Remark \ref{Rem:EFKY}.
\begin{corollary}\label{cor:stablefinmt}
    Under the hypotheses of Theorem \ref{thm:stablefinner}, there exists a constant $C>0$ such that if \eqref{geofinnercond} holds for a measurable set $A\subseteq \Omega$ and some $0\leq \delta<1$, then \eqref{mtb} holds for some cartesian product set $\mathcal{P}$.
\end{corollary}
\begin{proof}
    Since \eqref{geofinnercond} implies
    \[
      \mu(A)\ge (1-\delta)\mu(A)^{\frac 12}\prod_{j=1}^m\mu_{S_j}(P_j(A))^{\frac 1{2q_j}},
    \]
    we have
    \begin{equation}
\label{geofincond:modified}
         \mu(A)\ge (1-\delta)\prod_{j=1}^{m+1}\mu_{S_j}(P_j(A))^{\frac 1{\tilde q_j}},
    \end{equation}
    where $\tilde q_j= 2{q_j}$ for $j=1,\hdots,m$, $\tilde q_{m+1}=2$, and $S_{m+1}=\{1,\hdots,d\}$. 
Thus, it suffices to show that \eqref{geofincond:modified} implies \eqref{mtb}.

We now consider the special choice $g_j=1_{P_j(A)}$ for a measurable set $A\subseteq \Omega$. 
Then it is straightforward to verify that $(g_j)$ satisfies \eqref{finnerstablecond} on replacing $q_j$ by $\tilde q_j$.
    Since 
    \[
    \sum_{j=1}^{m+1}\frac{1}{\tilde q_j}1_{S_j}\equiv 1,
    \]
Theorem \ref{thm:stablefinner} shows that there are nonnegative functions $h_i$ satisfying \eqref{finnerstablest}, again on replacing $q_j$ by $\tilde q_j$. We use Fubini's identity to rewrite this inequality as 
\begin{equation}
    \label{mtfineq}
\int \int\Big( (g_j\circ P_j)^{\frac {\tilde q_j}2}-c_j^{\frac 12}\prod_{i'\in S_j}h_{i'}^{\frac 12}\Big)^{2}\mathrm d\mu_i\mathrm d\mu_{S_j\setminus \{i\}} \lesssim \delta c_j\prod_{i'\in S_j}\|h_{i'}\|_1,
\end{equation}
where $i\in S_j$.
Since the $g_j$ are characteristic functions, there exist measurable sets $E_i\subseteq \Omega_i$ for $i=1,\hdots,d$ such that 
\begin{equation}
    \label{mtfineq2}
\|h_i^{\frac 12}-c_i'1_{E_i}\|_2\lesssim \delta^{\frac 12} \|h_i\|_1^{\frac 12},
\end{equation}
for some constants $c_i'>0$. Indeed, if the above fails for some $i$, then the uniform pointwise inequality
\[
\int\Big( (g_j\circ P_j)^{\frac {\tilde q_j}2}-c_j^{\frac 12}\prod_{i'\in S_j}h_{i'}^{\frac 12}\Big)^{2}\mathrm d\mu_i\gtrsim \delta c_j\|h_{i}\|_1\prod_{i'\in S_j\setminus\{i\}}h_{i'}
\]
must hold,
which contradicts \eqref{mtfineq} on integrating with respect to $\mu_{S_j\setminus \{i\}}$.

Combining \eqref{mtfineq} and \eqref{mtfineq2}, we conclude that
\[
\Big\|g_{m+1}-c_{m+1}^{\frac 12}\prod_{i=1}^dc_i'1_{E_i}\Big\|_2\lesssim \delta^{\frac 12}\|g_{m+1}\|_{\tilde q_{m+1}}^{\frac {\tilde q_{m+1}}2}.
\]
Since $g_{m+1}=1_A$, this gives the desired inequality \eqref{mtb}.
\end{proof}

\section{Stability for the adjoint Finner inequality}\label{sec:stadfi}
Here we consider stability for the adjoint Finner inequality \eqref{af}, and prove Theorem \ref{thm:adfstable2}. In \cite{BT} it was shown that equality in the adjoint Loomis--Whitney inequality occurs if and only if $f$ is of the form $f(x)=\prod_{i=1}^dh_i(x_i)$, where each $h_i: \Omega_i\rightarrow \mathbb R_+$ is a constant multiple of a characteristic function. The proof in \cite{BT} proceeds in two steps. First, it is shown that any extremiser $f$ is of the form $f(x)=\prod_{i=1}^dh_i(x_i)$ for nonnegative measurable functions $h_1,\hdots,h_d$. Second, by replacing $f$ with a (tensor) product of functions, it is deduced that  each $h_i$ is a characteristic function up to a positive constant multiple. The first step relies on the Loomis--Whitney and H\"older inequalities. Thus Theorems \ref{thm:stablefinner} and \ref{thm:aldaz}, which provide stable versions of these inequalities, indicate that a stable version of this step should be possible. 
\begin{comment}
However, the second step turns out to be more delicate. For example, if
    \[
    g(x)=\prod_{i}|x_i|^{-\frac 1{p+}}1_{[-1,1]}(x_i)
    \]
    for some $p<p_+<1$ then $\|f+\varepsilon g\|_{L^p(\mathbb{R}^d)}$ is finite, yet
%the expression $\prod_{j=1}^m\|(P_j)_*f\|_{L^{p_j}(\pi_{S_j})}^{\theta_j}$ is finite, whereas
    \[
    \prod_{j=1}^m\|(P_j)_*(f+\varepsilon g)\|_{L^{p_j}(\pi_{S_j})}^{\theta_j}
    \]
    is infinite for any $f\in L^p(\mathbb{R}^d)$ satisfying \eqref{ineq:stableadfcond}, for all $\varepsilon>0$. This makes \eqref{ineq:stableadfcond} unstable under perturbations with respect to the $L^p$ metric. Thus, the argument sketched above does not readily admit a stable version.    
\end{comment}
However, the second step turns out to be more delicate. In order to modify this second step, we make use of the following elementary consequence of Markov's inequality, which provides suitable pointwise estimates.

\begin{lemma}\label{lem:f-g}
    Let $N>0$, $S\subset \{1,\hdots,d\}$, and $f,g:\Omega_S \rightarrow \mathbb R_+$. If $\|f-g\|_2\le \varepsilon^{N+1}\|f\|_2$, then there exists a measurable set $A\subset \Omega_S$ such that $f\neq0$ on $A$, $\|f1_{A^c}\|_2+\|g1_{A^c}\|_2\lesssim \varepsilon^N\|f\|_2$, and $|f-g|\le\varepsilon |f|$ on $A$.
\end{lemma}
\begin{proof}
    Define
    \[
    A:=\{x\in \Omega_S:f(x)\neq0, |1-g/f|\le \varepsilon\}.
    \]
    Then it directly follows that $|f-g|\le\varepsilon |f|$ on $A$. Furthermore, one can observe that
    \[
    \|\varepsilon f1_{A^c}\|_2\le \|(f-g)1_{A^c}\|_2\le\|f-g\|_2\le \varepsilon^{N+1}\|f\|_2. 
    \]
    Thus, we have $\|f1_{A^c}\|_2\le \varepsilon^{N}\|f\|_2$. Since $\|g1_{A^c}\|_2\le \|(f-g)1_{A^c}\|_2+\|f1_{A^c}\|_2$, we obtain that $\|g1_{A^c}\|_2\lesssim \varepsilon^{N}\|f\|_2$, which gives the desired result.
\end{proof}
\begin{remark}\label{Rem:convert} A useful feature of Lemma \ref{lem:f-g} is that it allows $L^2$ closeness to be converted into an $L^p$ closeness for any $0<p<\infty$. More precisely, if $\|f-g\|_2\le \varepsilon \|f\|_2$, then for any $a>0$,
\[
\begin{aligned}
    \|f^a-g^a\|_{\frac 2a}&\lesssim_a \|f^a1_{A^c}\|_{\frac 2a}+\|g^a1_{A^c}\|_{\frac 2a}+\varepsilon^{\frac 12}\|f^a\|_{\frac 2a}\\
    &\lesssim_a (\varepsilon^{\frac a2}+\varepsilon^{\frac 12})\|f^a\|_{\frac 2a}. 
\end{aligned}
\]
In this way $L^2$ stability statements are seen to imply certain $L^p$ stability statements; see, for example, Remark \ref{Rem:CD}.
\end{remark}
%Motivated by this remark, we first establish the following statement, which implies \eqref{notgoodenough}.
%our approach to Theorem \ref{thm:adfstable2} is to first establish the following ostensibly weaker statement.
With Remark \ref{Rem:convert} in mind, we begin by establishing the following $L^2$-based stability statement for the adjoint Finner inequality. As we shall see, this is a natural first step towards a proof of Theorem \ref{thm:adfstable2}, given the $L^2$ nature of Theorems \ref{thm:stablefinner} and \ref{thm:aldaz}.
\begin{theorem}\label{thm:adFstable}
    Under the hypotheses of Theorem \ref{thm:adfstable2}, for $\varepsilon>0$, there exists $\delta>0$ such that \eqref{ineq:stableadfcond} implies that
    \begin{equation}
        \label{ineq:stableadfest}
    \|f^{\frac p2}-c^{\frac p2}1_{E_1\times\cdots\times E_d}\|_2\le \varepsilon \|f\|_p^{\frac p2},
    \end{equation}
    for some constant $c>0$ and measurable sets $E_i\subset \Omega_i$.
\end{theorem}
\begin{proof} Our proof is guided by the argument in \cite[Theorem 2.1]{BT}, drawing on the stability Theorems \ref{thm:aldaz}, \ref{thm:stablefinner}, along with Lemma \ref{lem:f-g}.
    
    Set $\theta_*=\min_j\theta_j$ and
        \[
    g_j=((P_{j})_*f)^{\frac{p_j}{q_j}}.
    \]
    We begin by writing
    \[
    \begin{aligned}
    \|f\|_p&=\Big(\int \prod_{j=1}^m(f^{\theta_jp}(x)g_j^{p-1}(P_{j}x))\prod_{j=1}^mg_j^{1-p}(P_{j}x)\mathrm{d}\mu(x)\Big)^{\frac 1p}\\
    &\le \Big(\int \prod_{j=1}^mf^{\theta_j}g_j^{\frac{p-1}{p}}(P_{j}x)\mathrm{d}\mu(x)\Big)\Big(\int \prod_{j=1}^mg_j(P_{j}x)\mathrm{d}\mu(x)\Big)^{\frac{1-p}p}\\
    &\le \prod_{j=1}^{m}\|(P_{j})_*f\|_{L^{p_j}(\Omega_{S_j})}^{\theta_j};
    \end{aligned}
    \]
        here the first inequality follows from H\"older’s inequality, and for the second we apply H\"older’s inequality to the first factor and Finner's inequality to the second, noting that
        \[
        \|f^{\theta_j}g_j^{\frac{p-1}p}\|_{\theta_j^{-1}}=\|(P_j)_*f\|_{p_j}^{p_j\theta_j}, \quad  \|g_j\|_{q_j}=\|(P_j)_*f\|_{p_j}^{\frac {p_j}{q_j}},
        \]
and
    \[
    p_j\theta_j+\frac{p_j}{q_j}\left(\frac 1p-1\right)=\theta_j.
    \]
    By \eqref{ineq:stableadfcond} it follows that 
      \begin{equation}
          \label{LWstablefromad}
            \Big(\int \prod_{j=1}^mg_j(P_{j}x)\mathrm{d}\mu(x)\Big)^{\frac{1-p}p}\ge (1-\delta)\prod_{j=1}^m\|(P_{j})_*f\|_{L^{p_j}(\Omega_{S_j})}^{\frac{p_j(1-p)}{q_jp}},
      \end{equation}
            \begin{equation}
            \nonumber
            \begin{aligned}
                \Big(\int \prod_{j=1}^m&(f^{\theta_jp}(x)g_j^{p-1}(P_{j}x))\prod_{j=1}^mg_j^{1-p}(P_{j}x)\mathrm{d}\mu(x)\Big)^{\frac 1p}\\
    &\ge (1-\delta)\Big(\int \prod_{j=1}^mf^{\theta_j}g_j^{\frac{p-1}{p}}(P_{j}x)\mathrm{d}\mu(x)\Big)\Big(\int \prod_{j=1}^mg_j(P_{j}x)\mathrm{d}\mu(x)\Big)^{\frac{1-p}p},
            \end{aligned}
            \end{equation}
            and 
    \begin{equation}
        \nonumber
\int \prod_{j=1}^mf^{\theta_j}g_j^{\frac{p-1}{p}}(P_{j}x)\mathrm{d}\mu(x)\ge (1-\delta)\prod_{j=1}^{m}\|(P_{j})_*f\|_{L^{p_j}(\Omega_{S_j})}^{p_j\theta_j}.
    \end{equation}
     By the stability of the Finner and H\"older inequalities, Theorems \ref{thm:stablefinner} and \ref{thm:aldaz}, the functions in the above inequalities are close to their respective extremisers.    
    More precisely, for $\varepsilon>0$, there exists $\delta>0$ such that if \eqref{ineq:stableadfcond} holds, then there exist positive constants $c_f,c_j,c_{j,\ell}$ and nonnegative functions $x_i\mapsto h_i(x_i)$ satisfying
        \begin{equation}
    \label{ineq:Fstable1}
                \Big\|g_j^{\frac{q_j}{2}}-c_j^{\frac 12}\prod_{i\in S_j}h_i^{\frac 12}\Big\|_{2}\le  \varepsilon^{2} \|g_j\|^{\frac{q_j}{2}}_{q_j}=\varepsilon^{2} \|(P_{j})_*f\|_{L^{p_j}(\Omega_{S_j})}^{\frac{p_j}{2}},
    \end{equation}
    \begin{equation}
        \label{ineq:fgjstableFinner}
    \begin{aligned}
    \Big\|f^{\frac 12}\prod_{j=1}^m(g_j\circ P_j)^{\frac{p-1}{2p}} -c_f\prod_{j=1}^m(g_j\circ P_j)^{\frac{1}{2}}\Big\|_2&\le  \varepsilon^{\frac 1p} \Big\|f^{\frac 12}\prod_{j=1}^m(g_j\circ P_j)^{\frac{p-1}{2p}}\Big\|_2\\
    &\le  \varepsilon^{\frac 1p} \prod_{j=1}^m\|(P_{j})_*f\|_{L^{p_j}(\Omega_{S_j})}^{\frac {p_j\theta_j}{2}},    
    \end{aligned}
    \end{equation}
    and for all $j\neq\ell$,
    \begin{equation}
        \label{ineq:gjgellFinner}
    \|f^{\frac 12}(g_j\circ P_j)^{\frac{p-1}{2\theta_jp}}-f^{\frac 12}(c_{j,\ell}g_\ell\circ P_\ell)^{\frac{p-1}{2\theta_\ell p}}\|_2\le \varepsilon^{\frac{\theta_*^{-1}}p+1} \|(P_j)_*f\|_{L^{p_j}(\Omega_{S_j})}^{\frac {p_j}2};
    \end{equation}
the relevance of the particular powers of $\varepsilon$ chosen here will become clear shortly. Here, in a slight abuse of notation, $\prod_{i\in S_j}h_i$ is to be interpreted as the tensor product function $\Omega_{S_j}\ni P_jx\mapsto\prod_{i\in S_j}h_i(x_i)$.

Recalling the discussion at the beginning of this section, Step 1 in our argument is to show that
\begin{equation}
    \label{ineq:fHfinner}
\|f^{\frac p2}-H^{\frac 12}\|_2\lesssim  \varepsilon \|f\|_p^{\frac p2}
\end{equation}
%\footnote{JB: Wouldn't it be more natural to write $\|f^{p/2}-H^{1/2}\|_2$ here, and take the square roots of the definition of $H$ below?;I think so too! Thanks.} 
where
\[            H(x)=c_f^{2p}\prod_{j=1}^mc_j^{\frac 1{q_j}}\prod_{i=1}^dh_i(x_i).\]
The estimate \eqref{ineq:fHfinner} is natural to expect, in the sense that \eqref{ineq:Fstable1} and \eqref{ineq:fgjstableFinner} indicate that $f^{p/2}$ is close to $c_f^{p}\prod_j (g_j\circ P_j)^{1/2}$, while each $g_j^{1/2}$ is close to $(c_j\prod_{i\in S_j}h_i)^{1/2q_j}$.
%\footnote{JB: A very pedantic point probably, but are we clear enough about what we mean by $\prod_{i\in S_j}h_i$? This has a different interpretation from $\prod (g_\ell\circ P_\ell)$, for example, being really a tensor product. Such expressions come up again several times.; I add a comment about this below \eqref{ineq:gjgellFinner}. Please check... JB: I have changed it a bit - please check.; It looks much better. Thanks!}
    Indeed, by the reverse H\"older and Finner inequalities,
    \[
    \begin{aligned}
        \Big\|f^{\frac 12}\prod_{j=1}^m(g_j\circ P_j)^{\frac{p-1}{2p}} -c_f\prod_{j=1}^m(g_j\circ P_j)^{\frac{1}{2}}\Big\|_2^2&\ge\Big\|\Big(f^{\frac 12}-c_f\prod_{j=1}^m(g_j\circ P_j)^{\frac{1}{2p}}\Big)^{2}\Big\|_{p}\Big\|\prod_{j=1}^m(g_j\circ P_j)^{\frac{p-1}{p}}\Big\|_{\frac p{p-1}}\\
        &\ge\Big\|\Big(f^{\frac 12}-c_f\prod_{j=1}^m(g_j\circ P_j)^{\frac{1}{2p}}\Big)^{2}\Big\|_{p}\prod_{j=1}^m\|(P_{j})_*f\|_{L^{p_j}(\Omega_{S_j})}^{\theta_j(p_j-1)};
    \end{aligned}
            \]
            here we have used \eqref{relation:thetajpj} and the fact that $p-1<0$. Combining \eqref{ineq:stableadfcond} and \eqref{ineq:fgjstableFinner}, we obtain
            \[
            \Big\|\Big(f^{\frac 12}-c_f\prod_{j=1}^m(g_j\circ P_j)^{\frac{1}{2p}}\Big)^{2}\Big\|_{p}\le  \varepsilon^{\frac 2p} \prod_{j=1}^m\|(P_{j})_*f\|_{L^{p_j}(\Omega_{S_j})}^{\theta_j}\lesssim \varepsilon^{\frac 2p} \|f\|_p,
            \]
            and using the elementary inequality  $|A-B|\le |A^r-B^r|^{1/r}$, for $r\ge 1$ and $A, B>0$, this implies that
    \[
    \Big\|f^{\frac p2}-c_f^p\prod_{j=1}^m(g_j\circ P_j)^{\frac 1{2}}\Big\|_2\lesssim    \varepsilon\|f\|_p^{\frac p2}.
    \]
    Moreover, combining this with \eqref{LWstablefromad}, we also have that 
            \begin{equation}
                \nonumber
                \|f\|_p^{\frac p2} \sim 
            c_f^p \Big\|\prod_{j=1}^m(g_j\circ P_j)^{\frac {1}{2}}\Big\|_2\sim c_f^p \prod_{j=1}^m\|(P_{j})_*f\|_{L^{p_j}(\Omega_{S_j})}^{ \frac {p_j}{2q_j}}.
            \end{equation}
            On the other hand, by the same argument as in the proof of \eqref{Fstable2}, \eqref{ineq:Fstable1} yields that
            \begin{equation}
                \nonumber
            \begin{aligned}
                \Big\|\prod_{j=1}^m(g_j\circ P_j)^{\frac{1}{2}}-\prod_{j=1}^m(H_j\circ P_j)^{\frac{1}{2q_j}}\Big\|_2\lesssim   \varepsilon^2\prod_{j=1}^m\|(P_{j})_*f\|_{L^{p_j}(\Omega_{S_j})}^{\frac {p_j}{2q_j}},
            \end{aligned}
            \end{equation}
            where $H_j(P_jx)=c_j\prod_{i\in  S_j}h_i(x_i)$; note that $
            H=c_f^{2p}\prod_{j=1}^m(H_j\circ P_j)^{{1}/{q_j}}$.
            Combining the above three inequalities, we conclude \eqref{ineq:fHfinner}.

    We now turn to the second step, where Lemma \ref{lem:f-g} will be useful. In order to prove \eqref{ineq:stableadfest}, it suffices to show that 
    \begin{equation}
        \label{ineq:H-c1finner}
    \|H^{\frac 12}-c1_{E_1\times\cdots\times E_d}\|_2\lesssim \varepsilon^{\frac 12} \|H^{\frac 12}\|_2
    \end{equation}
    for some constant $c>0$ and measurable sets $E_i\subset \Omega_i$. Indeed, after replacing $\varepsilon$ by $(\varepsilon/C)^{2}$ for suitable constant $C>1$, \eqref{ineq:stableadfest} follows from \eqref{ineq:fHfinner} and \eqref{ineq:H-c1finner}.  

Applying Lemma \ref{lem:f-g} to \eqref{ineq:gjgellFinner}, there exist measurable sets $A_{j,\ell}\subset \Omega$ such that $f>0$ on $A_{j,\ell}$,
\begin{align}
    \label{ineq:lemma2.22.6}
\|f^{\frac 12}(g_j\circ P_j)^{\frac{p-1}{2\theta_jp}}1_{A_{j,\ell}^c}\|_2\lesssim \varepsilon ^{\frac{\theta_*^{-1}}{p}}\|(P_j)_*f\|_{p_j}^{\frac {p_j}2},
\end{align}
and 
\[
f^{\frac 12}(g_j\circ P_j)^{\frac{p-1}{2\theta_jp}}=(1+O(\varepsilon))f^{\frac 12}(c_{j,\ell}g_\ell\circ P_\ell)^{\frac{p-1}{2\theta_\ell p}} \quad on \quad A_{j,\ell}.
\]
Set 
\[
 A:=\bigcap_{j\neq \ell}A_{j,\ell}.
\]
Since $f>0$ on $A$, we have 
 \begin{equation}
     \nonumber
   g_j(P_jx)^{\frac 1{\theta_j}}=(1+O(\varepsilon))(c_{j,\ell}g_\ell(P_\ell x))^{\frac 1 {\theta_\ell}}
 \end{equation}
    for $x\in  A$  and $\ell\neq j$. By \eqref{relation:thetajpj}, this implies that
\begin{equation}
    \label{ineq:gjqjglql}
g_j(P_jx)^{q_j}=(1+O(\varepsilon))(c_{j,\ell}g_\ell(P_\ell x))^{q_\ell} \quad on \quad A.
\end{equation}
Recalling that $\sum_j\theta_j=1$, the estimate \eqref{ineq:lemma2.22.6}, combined with an application of H\"older's inequality, yields
\[
\begin{aligned}
    \Big\|f^{\frac 12}\prod_{j=1}^m(g_j\circ P_j)^{\frac{p-1}{2p}}1_{\cup_{j',\ell} A_{j',\ell}^c}\Big\|_2&\le\sum_{j', \ell}\prod_{j=1}^m\|f^{\frac 12}(g_j\circ P_j)^{\frac{p-1}{2\theta_jp}}1_{ A_{j',\ell}^c}\|_2^{\theta_j}\\
    &\lesssim \varepsilon^{\frac 1p}\prod_{j=1}^m\|(P_j)_*f\|_{p_j}^{\frac {p_j\theta_j}{2}};
\end{aligned}
\]
the gain coming from the $j=j'$ factor in the product for each $j'$.
%\footnote{JB: added}
On the other hand, by applying the reverse H\"older and Finner inequalities we have
\[    \Big\|f^{\frac 12}\prod_{j=1}^m(g_j\circ P_j)^{\frac{p-1}{2p}}1_{A^c}\Big\|_2^2\ge \|f1_{ A^c}\|_p\prod_{j=1}^m\|(P_j)_*f\|_{p_j}^{\theta_j(p_j-1)}.\]
Finally, combining the above two estimates with \eqref{ineq:stableadfcond}, we have
\begin{equation}
            \nonumber
            \|f1_{A^c}\|_p\lesssim \varepsilon^{\frac 2p}\prod_{j=1}^m\|(P_j)_*f\|_{p_j}^{\theta_j}\lesssim  \varepsilon^{\frac 2p}\|f\|_p.
\end{equation}
Together with \eqref{ineq:fHfinner}, this yields
\begin{equation}
            \label{ineq:ferrorfinner}
            \|f^{\frac p2}1_{A^c}\|_2+\|H^{\frac 12}1_{A^c}\|_{2}\lesssim \varepsilon \|H^{\frac 12}\|_2.
\end{equation}

Next, applying Lemma \ref{lem:f-g} to \eqref{ineq:Fstable1}, there exists a set $A_j\subset \Omega_{S_j}$ such that
    \begin{equation}
        \label{ineq:Herrorfinner}
    \|g_j^{\frac {q_j}2}\|_{L^2(\Omega_{S_j}\setminus  A_j)}+\Big\|c_j^{\frac 12}\prod_{i\in S_j}h_i^{\frac 12}\Big\|_{L^2(\Omega_{S_j}\setminus  A_j)}\lesssim  \varepsilon \|(P_j)_*f\|_{p_j}^{\frac {p_j}2}
    \end{equation}
    and 
    \begin{equation}
        \label{ineq:gjcjhisfinner}
    g_j^{\frac {q_j}2}=(1+O(\varepsilon))c_j^{\frac 12}\prod_{i\in S_j}h_i^{\frac 12} \quad \text{ on } A_j.
    \end{equation}
    Now set
    \[
    \tilde A:= A\cap\bigcap_{j=1}^m P_j^{-1}( A_j).
    \]
    Combining \eqref{ineq:ferrorfinner}, Finner's inequality and \eqref{ineq:Herrorfinner}, it follows that 
    \begin{equation}
        \label{ineq:Herrorfinner2}
    \|H^{\frac 12}1_{\tilde A^c}\|_2\lesssim  \varepsilon\|H^{\frac 12}\|_2.
    \end{equation}
    Further, by \eqref{ineq:gjqjglql} and \eqref{ineq:gjcjhisfinner}, we have
    \begin{equation}
        \label{ineq:hl=hjfinner}
        \prod_{i\in S_j}h_i(x_i)=b_{j,\ell}(1+O(\varepsilon))\prod_{i\in S_\ell}h_{i}(x_{i}), \quad \text{for }x\in \tilde A, j\not=\ell,
    \end{equation}
    for some constants $b_{j,\ell}>0$. 
The required inequality \eqref{ineq:H-c1finner} is now reduced to Lemma \ref{lem:prodchar1} below.
\end{proof}
\begin{lemma}\label{lem:prodchar1}
    Let $d\ge2$, $H(x)=\prod_{i=1}^dh_i(x_i)$, and $S_j$ satisfy \eqref{finnerrelation} and \eqref{assumptionM}. Suppose that $\tilde A\subset \Omega$ satisfies \eqref{ineq:Herrorfinner2} and \eqref{ineq:hl=hjfinner}. Then there exist  sets $E_i\subset \Omega_i$ and constants $c_i>0$ such that $\|h_i^{1/2}-c_i1_{E_i}\|_2\lesssim \varepsilon^{1/2}\|h_i\|_1^{1/2}$.
\end{lemma}

\begin{proof} 
By symmetry it suffices to show that $h_d$ satisfies the desired property.
By \eqref{finnerrelation} and \eqref{assumptionM} there exist $j_1,j_2$ such that $d\in S_{j_1}$ and $d\notin S_{j_2}$. On the one hand $M_d$ is nonempty by \eqref{finnerrelation}, so the existence of such a $j_1$ is immediate. On the other hand, if no such $j_2$ were to exist then $M_d=\{1,\hdots, m\}$, meaning that $\sum_{j\in M_i}q_j^{-1}=1$ for all $i$, implying that $M_i=\{1,\hdots, m\}$ for all $i$, contradicting \eqref{assumptionM}. 

    By \eqref{ineq:Herrorfinner2} and \eqref{ineq:hl=hjfinner} there exists a constant $c>0$ such that
    \[
    \|H^{\frac 12}-c H_{j_1,j_2}^{\frac 12}1_{\tilde A}\|_2\lesssim \varepsilon \|H^{\frac 12}\|_2,
    \]
    where 
    \[
    \begin{aligned}
         H_{j_1,j_2}(x)&=\prod_{i\notin S_{j_1}}h_i(x_i) \prod_{i\in S_{j_2}}h_i(x_i)\\
        &=\prod_{i\notin S_{j_1}\setminus S_{j_2}}h_i(x_i) \prod_{i\in S_{j_2}\setminus S_{j_1}}h_i(x_i).        
    \end{aligned}
    \]
    For simplicity, set $S^1=S_{j_1}\setminus S_{j_2}$
    %\footnote{JB: minor abuse of notation: perhaps we could use $S^1$ and $S^2$?} 
    and $S^2=S_{j_2}\setminus S_{j_1}$.
    Then the above inequality becomes
    \[
    \begin{aligned}
        \int  \prod_{i=1}^{d-1}h_i\int \Big(h_{d}^{\frac 12        }-c\frac{\prod_{i\in S^2}h_i^{1/2}}{\prod_{i\in S^1\setminus \{d\}}h_i^{1/2}}1_{\tilde A}\Big)^2\mathrm{d}\mu_d \mathrm{d}\mu_{I_{d-1}}
        \lesssim \varepsilon^2\prod_{i=1}^d\|h_i\|_1.
    \end{aligned}
    \]
    Define sets $O,\tilde O\subseteq\Omega_{I_{d-1}}$ by  
    \[
    O:=\Big\{(x_1,\hdots, x_{d-1}):\int \Big(h_{d}^{\frac 12}(x_{d})-c\frac{\prod_{i\in S^2}h_i^{1/2}(x_i)}{\prod_{i\in S^1\setminus \{d\}}h_i^{1/2}(x_i)}1_{\tilde A}\Big)^2\mathrm{d}\mu_d(x_{d})\le  \varepsilon \|h_d\|_1\Big\}
    \]
    and $\tilde O:=O\cap \big\{(x_1,\hdots,x_{d-1}):\prod_{i=1}^{d-1}h_i(x_i)>0\big\}$.    Since $$\Big\|\prod_{i=1}^{d-1} h_i1_{\Omega_{I_{d-1}}\setminus \tilde O}\Big\|_1\lesssim \varepsilon \prod_{i=1}^{d-1}\|h_i\|_1$$ by definition, it follows that $\tilde O$ is nonempty.
    Finally we fix $v=(v_1,\hdots,v_{d-1})\in \tilde O$ and define
    \[
    c_d=c\frac{\prod_{i\in S^2}h_i^{1/2}(v_i)}{\prod_{i\in S^1\setminus \{d\}}h_i^{1/2}(v_i)}
    \]
    and
    \[
    E_d:=\left\{x_d\in \Omega_d: (v,x_d) \in \tilde A \right\}. 
    \]
    It remains to observe that 
    \[
    \|h_d^{\frac 12}-c_d1_{E_d}\|_2^2
\le  \varepsilon \|h_d\|_1,
    \]
   as required.
    %By repeating the same argument with replacing $h_1$ by $h_i$, we can choose sets $E_i$ and constants $c_i$ satisfying \[\|h_i-c_i1_{E_i}\|_2\lesssim  \varepsilon^{1/2}\|h_i\|_2,\] for all $i=1,\hdots,d$.
\end{proof}

We now turn to the proof of Theorem \ref{thm:adfstable2}.

\begin{proof}[Proof of Theorem \ref{thm:adfstable2}]
The estimate \eqref{notgoodenough} follows directly from Theorem \ref{thm:adFstable} by Remark \ref{Rem:convert}, and so it remains to establish \eqref{ineq:pjfpj1}.

    In the proof of Theorem \ref{thm:adFstable} we observed that for each $\varepsilon>0$ there exists $\delta>0$ such that \eqref{ineq:stableadfcond} implies \eqref{ineq:Fstable1} and $\|h_i^{1/2}-c_i'1_{E_i}\|_2\lesssim \varepsilon \|h_i\|_1^{1/2}$. By the triangle inequality, these estimates give
    \begin{equation}
        \label{ineq:L2distancePjf}
        \Big\|((P_j)_*f)^{\frac {p_j}{2}}-  c_j^{\frac 12}\prod_{i\in S_j}c_i'1_{E_i}\Big\|_{L^2(\Omega_{S_j})}\lesssim  \varepsilon \|(P_j)_*f\|_{L^{p_j}({\Omega_{S_j}})}^{\frac{p_j}2}.
    \end{equation}
    By Lemma \ref{lem:f-g} (see Remark \ref{Rem:convert}), it follows that
      \begin{equation}
        \label{ineq:pjfpj12}
    \|(P_j)_*f- \tilde c_j(P_j)_*(1_{E_1\times\cdots \times E_d})\|_{L^{p
    _j}(\Omega_{S_j})}\lesssim  \varepsilon^{\frac 12} \|(P_j)_*f\|_{L^{p_j}({\Omega_{S_j}})},
    \end{equation}
    for some constants $\tilde c_j>0$. Here we have also used the fact that $$(P_j)_*(1_{E_1\times\cdots \times E_d})(P_j x)=\prod_{i\notin S_j}\mu_i(E_i)\prod_{i\in S_j}1_{E_i}(x_i).$$
It remains to prove that there is a single constant $c$ for which
  \begin{equation}
        \label{ineq:pjfpj13}
    \|(P_j)_*f- c(P_j)_*(1_{E_1\times\cdots \times E_d})\|_{L^{p
    _j}(\Omega_{S_j})}\lesssim  \varepsilon^{\frac {\theta_*}2} \|(P_j)_*f\|_{L^{p_j}({\Omega_{S_j}})}
    \end{equation}
for all $j$, where $\theta_*=\min_j\theta_j$. Indeed, \eqref{ineq:pjfpj1} follows on replacing $\varepsilon$ by $(\varepsilon/C)^{2\theta_*^{-1}}$ for a suitable $C>1$.

By Theorem \ref{thm:adFstable}, there is a constant $c>0$ such that \eqref{ineq:stableadfest} holds. Note that the $L^2$ estimates \eqref{ineq:stableadfest} and \eqref{ineq:L2distancePjf} yield
\begin{equation}
    \label{ineq:flppjflp}
    \begin{aligned}
        \|f\|_p&=(1+O(\varepsilon ))\|c1_{E_1\times\cdots\times E_d}\|_p, \\
        \|(P_j)_*f\|_{p_j}&=(1+O(\varepsilon^{\frac 12}))\|\tilde c_j(P_j)_*(1_{E_1\times \cdots\times E_d})\|_{p_j},
    \end{aligned}
\end{equation}
for all $j$.
%After applying the change of variables $x\mapsto (|E_1|x_1,\hdots,|E_d|x_d)$, we may assume that $|E_i|=1$ for all $i$, and hence
%\footnote{JB: added. Is this what you mean here?; It was not, but I think this is very good point. I thought that reduction to (3.20) is automatic. So I did not explain the scaling for the reduction. The scaling is used very crucially on the proof of inequality below (3.23). But I think your comment is helpful to understand this reduction.. $c=(1+O(\varepsilon))\|f\|_p$ and $\tilde{c}_j=(1+O(\varepsilon^{1/2}))\|(P_j)_*f\|_{p_j}$. 
To prove \eqref{ineq:pjfpj13} it therefore suffices to show that
\begin{equation}
    \label{ineq:c/tildecj}
\frac{c}{\tilde c_j}=1+O(\varepsilon^{\frac{\theta_*}2})
\end{equation}
for all $j$.  

To prove \eqref{ineq:c/tildecj} we apply the adjoint Finner inequality to a nonnegative function that is close to $f-c1_{E_1\times\cdots\times E_d}$.
To do this we find a set $A$ which is close enough to the set $E_1\times\cdots\times E_d$ and such that $f-\tilde c1_A$ is nonnegative for a suitable constant $\tilde c$.
By Lemma \ref{lem:f-g} along with \eqref{ineq:stableadfest}, there exists a set $A\subseteq E_1\times \cdots\times E_d$ satisfying 
\begin{equation}
    \label{ineq:f=conA}
    \begin{aligned}
        (1-\varepsilon^{\frac 12}) c^{\frac p2}\le f^{\frac p 2}\le&(1+ \varepsilon^{\frac 12})c^{\frac p2} \text{  on } A,\\
        \|f1_{A^c}\|_p&\lesssim \varepsilon^{\frac 1 p}\|f\|_p,
    \end{aligned}
\end{equation}
and 
\begin{equation}
    \label{ineq:diffe-a}
\mu({E_1\times\cdots\times  E_d}\setminus A) \lesssim \varepsilon\mu(E_1\times\cdots\times E_d),
\end{equation}
which follows from $\|c^{ p/2}1_{E_1\times\cdots\times  E_d\setminus A}\|_2\lesssim \varepsilon^{1/2}\|c^{p/2}1_{E_1\times \cdots\times E_d}\|_2$. 
By taking
\[
\tilde c= c(1-\varepsilon^{\frac 12})^{\frac 2p},
\]
we have $f\ge \tilde c$ on $A$. Additionally, \eqref{ineq:f=conA} implies that
\begin{equation}
    \label{ineq:f-c1A}
    \begin{aligned}
        \|f-\tilde c1_A\|_p&\lesssim\|f1_{A^c}\|_p+\|(f-\tilde c)1_A\|_p\\
        &\lesssim \varepsilon^{\frac 1p}\|f\|_p+ \varepsilon^{\frac 12}c\|1_{E_1\times \cdots E_d}\|_p\\
        &\lesssim \varepsilon^{\frac 12}\|f\|_p.
    \end{aligned}
\end{equation}

We first treat the simpler case in which $\tilde c_j>\tilde c$ for all $j$. By the definition of $\tilde c$, this condition implies  that \begin{equation}\label{inter6}
\frac{c}{\tilde c_j}\le 1+O(\varepsilon^{1/2} ).
\end{equation}
On the other hand, since $f$ is a near extremiser of the adjoint Finner inequality, \eqref{ineq:flppjflp} gives that
\[
 c=(1+O(\varepsilon^{\frac 12}))\prod_{j=1}^m\tilde c_j^{\theta_j},
\]
which may be written as 
\[
1-C\varepsilon^{\frac 12}\le \prod_{j=1}^m\left(\frac c{\tilde c_j}\right)^{\theta_j}\le 1+C\varepsilon^{\frac 12}
\]
for some $C$; here we have used the fact that $\theta_1+\cdots+\theta_m=1$. By \eqref{inter6} we now also have $c/\tilde c_j\ge1-O(\varepsilon^{1/2})$, and  \eqref{ineq:c/tildecj} follows.

Turning to the remaining case, suppose now that 
$
\tilde c_j\le \tilde c
$
for some $j$, and let $j^*$ be such that $\tilde c_{j_*}=\min_j \tilde c_j$.
Arguing as in the previous case, in order to establish \eqref{ineq:c/tildecj} it is enough to show that
\begin{equation}\label{inter8}
\frac{c}{\tilde c_{j_*}}\le 1+O(\varepsilon^{\frac{\theta_*}{2}}).
\end{equation}
Now, since $\tilde c_{j_*}\le \tilde c$, the difference $f-\tilde c_{j_*}1_A$ is nonnegative, and so the adjoint Finner inequality (which requires nonnegativity) implies that
\begin{equation}
    \label{ineq:adfinnerrevised}
\|f-\tilde c_{j_*}1_A\|_p\le \prod_{j=1}^m\|(P_j)_*(f-\tilde c_{j_*}1_A)\|_{p_j}^{\theta_j}.
\end{equation}
To estimate the right hand side of this inequality, we observe that, by H\"older's inequality and \eqref{ineq:diffe-a},
\[
\begin{aligned}
\|(P_j)_*(1_{E_1\times\cdots\times E_d}-1_A)\|_{p_j}^{p_j}&=\int \Big|\int1_{E_1\times\cdots\times E_d\setminus A}\mathrm {d}\mu_{S_j^c}\Big|^{p_j}\prod_{i\in S_j}1_{E_i} \mathrm{d} \mu_{S_j}\\
&\le \mu(E_1\times\cdots\times E_d\setminus A)^{p_j}\prod_{i\in S_j}\mu_i(E_i)^{1-p_j} \\
&\lesssim \varepsilon^{p_j}\prod_{i'\notin S_j}\mu_{i'}(E_{i'})^{p_j}\prod_{i\in S_j}\mu_i(E_i)\\
&=\varepsilon^{p_j}\|(P_j)_*(1_{E_1\times\cdots\times E_d})\|_{p_j}^{p_j}.
%&=\varepsilon^{p_j}\|(P_j)_*(1_{E_1\times\cdots\times E_d})\|_{p_j}^{p_j}. 
\end{aligned}
\]
Using this along with \eqref{ineq:pjfpj12} it follows that
\[
\begin{aligned}
    \|(P_j)_*(f-\tilde c_{j_*}1_A)\|_{p_j}&\lesssim \|(P_j)_*(f-\tilde c_{j_*}1_{E_1\times\cdots\times E_d})\|_{p_j}+\|(P_j)_*(\tilde c_{j_*}1_{E_1\times\cdots\times E_d\setminus A})\|_{p_j} \\
    &\lesssim 
    \begin{cases}
    \varepsilon^{\frac 12}     \|(P_j)_*f\|_{p_j}, \quad &j=j_*, \\
    \|(P_j)_*f\|_{p_j}, \quad &j\neq j_*.
    \end{cases}
\end{aligned}
\]
Hence by \eqref{ineq:adfinnerrevised}  we have
\[
\|f-\tilde c_{j_*}1_A\|_p\lesssim \varepsilon^{\frac{\theta_{j_*}}2}\prod_{j=1}^m\|(P_j)_*f\|_{p_j}^{\theta_j}\lesssim \varepsilon^{\frac{\theta_{*}}2}\|f\|_p.
\]
Combining this with \eqref{ineq:f-c1A} it follows that $\|(\tilde c-\tilde c_{j_*})1_A\|_p\lesssim \varepsilon^{{\theta_*}/{2}}\|f\|_p\lesssim \varepsilon^{{\theta_*}/{2}}\|c1_{E_1\times\cdots E_d}\|_p$. Consequently, together with \eqref{ineq:diffe-a}, we obtain   
$\tilde c/\tilde c_{j_*}=1+O(\varepsilon^{\theta_*/2})$, which implies \eqref{inter8} by the definition of $\tilde c$.
\end{proof}

\section{The proofs of Theorems \ref{thm:BTlog} and \ref{thm:main}}\label{Sect:preliminaries}
In this section we establish Theorems \ref{thm:BTlog} and \ref{thm:main} via the adjoint Finner inequality and its stability (Theorem \ref{thm:adfstable2}). Here we will be working in the Euclidean setting with $\Omega_i=\mathbb{R}$ and $\mu_i$ being Lebesgue measure for each $i$. The product $\Omega_{S_j}$ is then a coordinate subspace of $\mathbb{R}^d$, and $\mu_{S_j}$ is its Lebesgue measure.

For $j=1,\hdots, d$, define sets of $d-k$ elements,
\[
S_j=\{i\in \{1,\hdots,d\}:(\{i\}+d\mathbb Z) \cap \{j,j+1,\hdots, j+k-1\}=\emptyset\},
\]
and set for $j=1,\hdots,d$, 
\[
q_j=d-k, \quad \theta_j=\frac 1d.
\]
Then we have the corresponding adjoint Finner inequality
\begin{equation}\label{alw}
\|f\|_{L^p(\mathbb R^d)}\le \prod_{j=1}^{d}\|(P_{j})_*f\|_{L^q(\Omega_{S_j})}^{\frac{1}{d}},
\end{equation}
where $p,q$ satisfy \eqref{scalingk} and $P_j$ is the orthogonal projection of $\mathbb{R}^d$ onto $\Omega_{S_j}$. 

Now, since the pushforward $(P_j)_*f:\Omega_{S_j}\rightarrow\mathbb{R}_+$ is given 
%is defined via duality by $$ \int_{H_j} (B_j)_*f(y)g(y)\mathrm{d}y=\int_H f(x)g(B_jx)\mathrm{d}x,$$ and 
by the marginal formula
$$
(P_j)_*f(y)=\int_{P_j^{-1}(y)}f,
$$
we may write
\[
(P_j)_*f(y)=T_{k,d}f(\pi_j,y),
\]
where $\pi_1,\hdots,\pi_d$ are the $k$-dimensional coordinate subspaces of $\mathbb{R}^d$ given by
\begin{equation}
    \label{def:pij}
\pi_j=\Omega_{S_j}^\perp=\langle e_j, e_{(j+1)\mod d},\hdots, e_{(j+k-1)\mod d}\rangle.
\end{equation}
Given \eqref{alw} it is now straightforward to establish Theorem \ref{thm:BTlog} by a minor adaptation of the proof of Theorem 8.7 in \cite{BT}.
\begin{proof}[Proof of Theorem \ref{thm:BTlog}]
Noting that 
$$\|(P_{j})_*f\|_{L^q(\pi_j^\perp)}=\|T_{k,d}f(\pi_j,\cdot)\|_{L^q(\pi_j^\perp)},$$
and taking logarithms, \eqref{alw} becomes
\begin{equation}\label{ablj}
\log\|f\|_{L^p(\mathbb R^d)}\le \frac{1}{d}\sum_{j=1}^{d}\log\|T_{k,d}f(\pi_j,\cdot)\|_{L^q(\pi_j^\perp)}.
\end{equation}
Replacing $f$ by $f\circ \rho$ in \eqref{ablj} for $\rho\in SO(d)$, and averaging, it follows that
\begin{equation}
    \label{abljrot}
\log\|f\|_{L^p(\mathbb R^d)}\le\int_{SO(d)}\frac 1d \sum_{j=1}^d \log\|T_{k,d}(f\circ \rho)(\pi_j,\cdot)\|_{L^q(\pi_j^\perp)}\mathrm{d}\rho,
\end{equation}
where $\mathrm{d}\rho$ is the uniform probability measure on $SO(d)$; here we have also used that $\|f\circ\rho\|_{L^p(\mathbb{R}^d)}=\|f\|_{L^p(\mathbb{R}^d)}$.
Finally we observe that
\begin{equation}
    \label{Tkdrho}
    \|T_{k,d}(f\circ\rho)(\pi,\cdot)\|_{L^q(\pi^\perp)}=\|T_{k,d}f(\rho^*\pi,\cdot)\|_{L^q((\rho^*\pi)^\perp)}
\end{equation}
for every $\pi\in\mathcal{G}_{k,d}$. The theorem now follows since 
$\mathrm{SO}(d)$ acts transitively on 
$\mathcal{G}_{k,d}$.
\end{proof}
Reflecting on the above argument, at least two strategies for improving the conclusion naturally present themselves. The first, and perhaps most evident, is suggested by the fact that \eqref{alw} (and thus \eqref{ablj}) cannot hold with equality for all rotations of $f$: the point being that extremisers to \eqref{alw} take the form $f=c\chi_{E_1\times\cdots\times E_d}$ for subsets $E_1,\hdots,E_d$ of the line and $c>0$ (see \cite[Theorem 2.1]{BT} and Theorem \ref{thm:adfstable2});
%{\color{red} here I'm just thinking about $k=1$, but there must be a similar statement for any $k$, no?}
 evidently this class is not rotation-invariant. The second strategy is to look to replace the Finner inequality with a different Brascamp--Lieb inequality that has adjoint constant less than $1$. In Section \ref{sec:gs}, we explore this second approach, and the role played by stability there.

In the remainder of this section we prove Theorem \ref{thm:main} using the first of these approaches. Roughly speaking, if $f$ is a near-extremiser for \eqref{Application of Finnerintro}, then it must be a near extremiser for \eqref{alw} for many rotations. However, Theorem \ref{thm:adfstable2} shows that such near extremisers are very far from being rotation-invariant, which leads to a contradiction. 

\subsection{Proof of Theorem \ref{thm:main}}
We begin by proving a quantitative statement showing that a near extremiser of \eqref{alw} is far from rotation-invariant. Since a rotation-invariant function $f$ satisfies $f=f\circ \rho$ for all rotations $\rho$, one may test for rotation-invariance by estimating the distance between $f$ and $f\circ \rho$. In view of the observation that a near extremiser of \eqref{alw} satisfies \eqref{ineq:stableadfest}, it reduces to establishing an estimate of the form
    \begin{equation}
        \label{ineq:P1P2}
    |\mathcal P_1\cap \rho\mathcal P_2|\le C\max\{|\mathcal P_1|,|\mathcal P_2|\},
    \end{equation}
    where $\mathcal P_1,\mathcal P_2$ are cartesian product sets in $\mathbb R^d$. We first verify that \eqref{ineq:P1P2} holds with $C<1$ for suitable rotations $\rho_{\pm}$, defined by
     \[
    \rho_{\pm}e_i=\begin{cases}
        \frac{\sqrt 2}{2}(e_1\pm e_d), &i=1, \\
        e_i, &i=2,\hdots,d-1, \\
        \frac{\sqrt 2}{2}(\mp e_1+e_d), &i=d.
    \end{cases}
    \]

\begin{lemma}\label{lem:P1P2}
 Suppose that $\rho=\rho_{\pm}$. Then 
 \begin{equation}
     \label{ineq:P1P2sharp}
     |\mathcal P_1\cap \rho\mathcal P_2|\le 2(\sqrt 2-1)|\mathcal P_1|^{\frac 12}|\mathcal P_2|^{\frac 12}
 \end{equation}
holds for all cartesian product sets $\mathcal P_1,\mathcal P_2$.
\end{lemma}
\begin{remark}
The inequality \eqref{ineq:P1P2sharp} is of Brascamp--Lieb type and directly implies \eqref{ineq:P1P2} with $C=2(\sqrt 2-1)$. Moreover, the constant $2(\sqrt 2-1)$ is sharp, as may be seen with the example $\mathcal{P}_1=\mathcal{P}_2=[-1,1]^d$. 
\end{remark}
\begin{proof}
    By the Brascamp--Lieb--Luttinger rearrangement inequality \cite[Theorem 1.2]{BLL}, we may assume that $\mathcal P_1,\mathcal P_2$ are rectangular boxes centered at the origin. Since $\rho_{\pm} e_j=e_j$ for $j=2,\hdots,d-1$, \eqref{ineq:P1P2sharp} reduces to the corresponding estimate in dimension $d=2$. Thus, it remains to show the desired inequality \eqref{ineq:P1P2sharp} for all rectangular boxes $\mathcal P_1,\mathcal P_2\subset \mathbb R^2$ when $\rho$ is a rotation through $\pi/4$. 

For $i=1,2,3,4$, set
\[
u_1=e_1, \quad  u_2=e_2, \quad  u_3=\frac{\sqrt2(e_1+e_2)}{2},  \quad u_4=\frac{\sqrt2(e_2-e_1)}{2},
\]
and define
\[
h_i:=\sup_{y\in \mathcal P_1\cap \rho\mathcal P_2}\langle y, u_i\rangle;
\]
this is the support function of $P_1\cap \rho\mathcal P_2$.
Without loss of generality, we may assume that $h_i=\sup_{y\in\mathcal P_1}\langle y, u_i\rangle$ for $i=1,2$ and $h_i=\sup_{y\in\rho\mathcal P_2}\langle y, u_i\rangle$ for $i=3,4$. Indeed, we may replace $\mathcal P_1$ and $\mathcal P_2$ by $[-h_1,h_1]\times [-h_2,h_2]$ and $[-h_3,h_3]\times[-h_4,h_4]$ respectively.

By applying the formula for the area of a polygon (see \cite{FI}), we have
\[
|\mathcal P_1\cap \rho\mathcal P_2|=2\sqrt2(h_1+h_2)(h_3+h_4)-2\sum_{i=1}^4h_i^2.
\]
Since $|\mathcal P_1|^{1/2}|\mathcal P_2|^{1/2}=4\sqrt{h_1h_2h_3h_4}$, it is enough to show that
\[
\sqrt2(h_1+h_2)(h_3+h_4)-\sum_{i=1}^4h_i^2-4(\sqrt2 -1)\sqrt{h_1h_2h_3h_4}\le 0. 
\]
To this end, define 
\[
A=\sqrt{h_1h_2}, \quad B=\sqrt{h_3h_4}, \quad x=h_1+h_2, \quad y=h_3+h_4,
\]
and set
\[
F_{A,B}(x,y):=\sqrt 2xy-x^2-y^2+2A^2+2B^2-4(\sqrt2-1)AB.
\]
It suffices to verify that
\begin{equation}\label{Fneg}
F_{A,B}(x,y)\le0
\end{equation}
whenever $x\ge2A>0$ and $y\ge 2B>0$.
Since $\nabla F_{A,B}\neq0$ for $x,y>0$, and $$\lim_{|(x,y)|\rightarrow \infty }F(x,y)=-\infty,$$ it 
is enough to establish \eqref{Fneg} on the boundary -- that is, when $x=2A$ or $y=2B$.
When $x=2A$, we have for $y\ge 2B$,
\[
\begin{aligned}
    F_{A,B}(2A,y)&=-(y-\sqrt2 A)^2+2B^2-4(\sqrt2-1)AB \\
    &\le \begin{cases}
        2B^2-4(\sqrt2-1)AB, & 2B\le \sqrt 2A, \\
        -2(A-B)^2, &2B\ge\sqrt 2 A. 
    \end{cases}
\end{aligned}
\]
For $2B\le \sqrt 2A$, one has $2B^2-4(\sqrt2-1)AB\le AB(4-3\sqrt 2)<0$, and so $F_{A,B}(2A,y)\le0$. By symmetry $F_{A,B}(x,2B)\le0$ for $x\geq 2A$, and so \eqref{Fneg} follows on the whole domain.
\end{proof}

By appealing to the continuity of the map $\rho\mapsto |\mathcal P_1\cap \rho\mathcal P_2|$, this is quickly seen to imply the following uniform version of Lemma \ref{lem:P1P2}.
\begin{lemma}\label{cor:intprod}
        Let $\mathcal P_1$ be a cartesian product set in $\mathbb R^d$. Then there exists a constant $c_d>0$, depending only on $d$, and a set $U\subset SO(d)$ with $\mathrm{d}\rho(U)>c_d$, such that  \eqref{ineq:P1P2} holds with $C=\sqrt 2-1/2$
    for all $\rho\in U$ and all cartesian product sets $\mathcal P_2$. Here $\mathrm{d}\rho$ is the uniform probability measure on $SO(d)$.
\end{lemma}
\begin{proof}
    As before, we may assume that $\mathcal P_1,\mathcal P_2$ are rectangular boxes, which are given by
    \[
    \begin{aligned}
        \mathcal P_1=\{x\in \mathbb R^d:|x_i|\le a_i\}, \quad 
        \mathcal P_2=\{x\in \mathbb R^d:|x_i|\le b_i\},
    \end{aligned}
    \]
    for $a_i,b_i>0$. After permuting the coordinate axes, we may assume that $a_1\ge \cdots \ge a_d$. 

       We first consider the case $a_1/a_d<16$. In this case, we have $a_1\sim a_i$ for all $i=2,\hdots,d$. Since  $2(\sqrt2-1)<C<1$, there exists $\varepsilon>0$ such that  $$(1+\varepsilon)^d-1+2(\sqrt 2-1)<C.$$  Let
       \[
       U_\varepsilon:=\big\{\rho'\in SO(d):|(\rho')_{i,j}|\le \frac \varepsilon{16d} \text{ for }i\neq j\big\}.
       \]
       Since $U_\varepsilon$ contains an open neighborhood of the identity, its measure is positive and depends only on $d$ and $\varepsilon$.
       Moreover, since $a_1/a_d<16$, one can observe that   
       \[
       |(\rho'x)_{i,i}|\le a_i+\sum_{j\neq i}\frac{\varepsilon a_j}{16d}\le (1+\varepsilon)a_i,
       \]
       for $\rho'\in U_\varepsilon$ and $x\in \mathcal P_1$. Thus, $\rho'\mathcal P_1$ is contained in $(1+\varepsilon )\mathcal P_1$. From this we conclude that for $\rho'\in U_\varepsilon$,
    \[\begin{aligned}
            |\rho'\mathcal P_1\cap \rho\mathcal P_2|&\le |\rho'\mathcal P_1\setminus \mathcal P_1|+|\mathcal P_1\cap \rho\mathcal P_2|\\
            &\le ((1+\varepsilon)^d-1)|\mathcal P_1|+|\mathcal P_1\cap \rho\mathcal P_2|.
    \end{aligned}
    \]
  By Lemma \ref{lem:P1P2}, for $\rho=\rho_\pm$, we have 
  \[
  |\mathcal P_1\cap (\rho')^*\rho\mathcal P_2|\le C\max\{|\mathcal P_1|,|\mathcal P_2|\},
  \]
  for all $\rho'\in U_\varepsilon$ and all $\mathcal P_2$. By choosing $U=\{(\rho')^*\rho\in SO(d):\rho=\rho_\pm, \rho'\in U_\varepsilon\}$, since $\mathrm d\rho(U)\ge\mathrm d\rho(U_\varepsilon)$, the desired result follows. 
  
    We now treat the case $a_1/a_d\ge 16$. Since the maps $\rho\mapsto \langle \rho e_1,e_1\rangle$ and $\rho\mapsto \langle \rho e_1,e_d\rangle$ are continuous, there exists a neighborhood $U\subset SO(d)$ of $\rho_{+}$, depending only on $d$,  upon which $\langle \rho e_1,e_1\rangle\ge 1/2 $ and $\langle \rho e_1,e_d\rangle\ge 1/2 $. For $\rho\in U$, considering that $\rho\mathcal P_2\subset \{x: |\langle \rho e_1,x\rangle| \le b_1 \}$ and $\langle \rho e_1,e_1\rangle\ge 1/2 $, we observe that for any $(x_2,\hdots, x_d)$,
    \begin{equation}\label{slow}
    |\{x_1\in \mathbb R: x\in \mathcal P_1\cap \rho\mathcal P_2\}|\le 4b_1,
    \end{equation}
    which implies that $|\mathcal P_1\cap \rho\mathcal P_2|\le 4b_1 \prod _{i=2}^d2a_i$.
    Thus, when $4b_{1}\le a_1$, \eqref{ineq:P1P2} with $C=1/2$ holds for all $\rho\in U$. Now assume that $b_{1}>a_1/4\ge 4a_d$. 
Similarly, since $\mathcal P_1\subset \{x:|x_d|\le a_d\}$ and $\langle \rho e_1,e_d\rangle\ge 1/2 $, we obtain that
    \[
    |\{\langle x,\rho e_1\rangle\in \mathbb R: x\in \mathcal P_1\cap \rho\mathcal P_2\}|\le 4a_d,
    \]
    provided that $\langle x, \rho e_i\rangle $ is fixed for $i=2,\hdots,d$.
Thus, it follows that $|\mathcal P_1\cap \rho\mathcal P_2|\le 4a_d \prod _{i=2}^d2b_i\le(1/2)\prod _{i=1}^d2b_i$, which yields the desired inequality \eqref{ineq:P1P2}  for all  $\rho\in U$.
\end{proof}

We are now ready to prove Theorem \ref{thm:main}.
%Let $\mathrm{d}\rho$ be the uniform probability measure on $SO(d)$. 
Suppose for a contradiction that the inequality \eqref{abljrot}
admits an extremising sequence. This means that
for all $\delta>0$ there exists a function $f\in L^p(\mathbb R^d)$ such that
\begin{equation}
    \label{ineq:lowerboundass}
\log\|f\|_{L^p(\mathbb R^d)}\ge -\delta^2+\int_{SO(d)}\frac 1d \sum_{j=1}^d \log\|T_{k,d}(f\circ \rho)(\pi_j,\cdot)\|_{L^q(\pi_j^\perp)}\mathrm{d}\rho,
\end{equation}
where $\pi_j$ is defined in \eqref{def:pij}.
Define the set
\[
U_*:=\Big\{\rho \in SO(d):\log\|f\|_{L^p(\mathbb R^d)}\ge -\delta+ \frac 1 d\sum_{j=1}^d \log\|T_{k,d}(f\circ \rho)(\pi_j,\cdot)\|_{L^q(\pi_j^\perp)}\Big\}.
\]
We claim that $U_*$ satisfies
\[
\mathrm{d}\rho(U_*)\ge 1-\delta.
\]
Indeed, by definition and \eqref{ablj}, it follows that
\[
\begin{aligned}
    \int_{SO(d)} \frac 1 d\sum_{j=1}^d\log\|T_{k,d}(f\circ \rho)(\pi_j,\cdot)\|_{L^q(\pi_j^\perp)}\mathrm{d}\rho&\ge\int_{U_*^c}\delta+\log\|f\|_{L^p(\mathbb R^d)}\mathrm{d}\rho+\int_{U_*}\log\|f\|_{L^p(\mathbb R^d)}\mathrm{d}\rho  \\
    &=\delta(1-\mathrm d\rho(U_*))+\log\|f\|_{L^p(\mathbb R^d)}.
\end{aligned}
\]
Thus, the assumption \eqref{ineq:lowerboundass} gives the desired bound for the size of $U_*$.

Now, for $\rho \in U_*$, we have
\[
\|f\circ \rho\|_{L^p(\mathbb R^d)}\ge e^{-\delta}\prod_{j=1}^d\|T_{k,d}(f\circ \rho)(\pi_j,\cdot)\|_{L^q(\pi_j^\perp)}^{\frac 1d}=e^{-\delta}\prod_{j=1}^d\|( P_{\pi_j^\perp})_*(f\circ \rho)\|_{L^q(\pi_j^\perp)}^{\frac 1d},
\]
which implies \eqref{ineq:stableadfcond} for the appropriate (adjoint Finner) data.
Therefore, by Theorem \ref{thm:adFstable}, for $\varepsilon>0$, if $\delta$ is sufficiently small then $f$ satisfies
\begin{equation}\label{use of Thm}
\|f\circ \rho-c_\rho1_{\mathcal P_\rho}\|_p\le \varepsilon \|f\|_p,
\end{equation}
for all $\rho \in U_*$, where $\mathcal P_\rho$ is a cartesian product set in $\mathbb R^d$. Moreover, for $\rho_1\in U_*$, Lemma \ref{cor:intprod} tells us that
\[
|\mathcal P_{\rho_1}\cap \rho\mathcal P_2|\le C\max\{|\mathcal P_{\rho_1}|,|\mathcal P_2|\}
\]
holds for all $\mathcal P_2$ and $\rho$ in a set  $U\subset SO(d)$, whose measure is positive. Thus, 
if $\delta$ is sufficiently small, there exists $\rho_2\in U_*$ such that the above inequality holds for $\rho=\rho_1^{-1}\rho_2$. By this, there exists a constant $c=(1-C)^{1/p}>0$ such that
\[
\|c_{\rho_1}1_{\mathcal P_{\rho_1}}\circ \rho_1^{-1}-c_{\rho_2}1_{\mathcal P_{\rho_2}}\circ \rho_2^{-1}\|_p\ge c\min_{j=1,2}\|c_{\rho_j}1_{\mathcal P_{\rho_j}}\|_p\ge c(1-\varepsilon^p)^{\frac 1p}\|f\|_p.
\]
This gives a contradiction when $\varepsilon$ is chosen sufficiently small, since the left-hand side of the above inequality is bounded above by $2^{1/p}\varepsilon \|f\|_p$ by two applications of \eqref{use of Thm}. Consequently, there exists a constant $\delta>0$ such that
\[
\log\|f\|_{L^p(\mathbb R^d)}\le -\delta^2+\int_{SO(d)} \frac 1d \sum_{j=1}^d\log\|T_{k,d}(f\circ \rho)(\pi_j,\cdot)\|_{L^q(\pi_j^\perp)}\mathrm{d}\rho,
\]
for all $f\in L^p(\mathbb R^d)$, which gives \eqref{ineqmain} with $D=\delta^2$ by \eqref{Tkdrho} and the fact that  $\mathrm{SO}(d)$ acts transitively on $\mathcal{G}_{k,d}$.

\section{$k$-plane estimates from Gaussian stability for the Brascamp--Lieb inequality}\label{sec:gs}

In this section we show how Theorem \ref{thm:main} may be reduced to a different stability problem in the wider theory of the Brascamp--Lieb inequalities, combined with an ingredient from frame theory. %We refer to \cite{BCCT,BT, BC:stable} for further context. 
We begin by recalling the relevant aspects of Brascamp--Lieb theory.

The Brascamp--Lieb inequalities, originating in \cite{BL:Y}, are a broad generalisation of a range of familiar functional inequalities in Euclidean analysis, including both Finner's inequality and Young's convolution inequality. They take the form
\begin{equation}\label{bl}
    \int_{H}\prod_{j=1}^m f_j(B_jx)^{c_j}\mathrm{d}x\leq \BL(\mathbf{B},\mathbf{c})\prod_{j=1}^m\left(\int_{H_j}f_j\right)^{c_j},
\end{equation}
where $B_j:H\rightarrow H_j$ is a linear map between Euclidean spaces, $c_j\in [0,1]$, and $f_j$ is a nonnegative integrable function on $H_j$, for each $1\leq j\leq m$. We refer to the tuple $(\mathbf{B},\mathbf{c})=((B_j),(c_j))$ as the \textit{Brascamp--Lieb datum}, and the optimal constant $\BL(\mathbf{B},\mathbf{c})$ as the \textit{Brascamp--Lieb constant}. Unlike in the special case of Finner's inequality \eqref{f}, where $c_j=1/q_j$ and $\BL(\mathbf{B},\mathbf{c})=1$, the maps $B_j$ are not required to have any coordinate structure. 
Necessary and sufficient conditions for the finiteness of $\BL(\mathbf{B},\mathbf{c})$ are well known (see \cite{BCCT}), with the simplest necessary condition
\begin{equation}\label{BLscaling}
\dim(H)=\sum_{j=1}^mc_j\dim H_j,
\end{equation}
arising from scaling considerations. Central to the theory of the Brascamp--Lieb inequalities is Lieb's theorem \cite{L}, which states that the optimal constant $\BL(\mathbf{B},\mathbf{c})$ in \eqref{bl} is saturated by centred Gaussian inputs $f_1,\hdots, f_m$.

Of particular relevance to us are the \textit{geometric Brascamp--Lieb data}, introduced by Ball and Barthe \cite{Ball,Barthe}. A datum $(\mathbf{B},\mathbf{c})$ is said to be \textit{geometric} if $B_jB_j^*$ is the identity on $H_j$ for each $j$ and
\begin{equation}\label{gblframe}
\sum_{j=1}^m c_jB_j^*B_j=I,
\end{equation}
the identity on $H$. Finner's inequality \eqref{f} also fits within this geometric framework, and it was shown in \cite{Ball,Barthe} that $\BL(\mathbf{B},\mathbf{c})=1$ holds for all geometric data. Given this fact, it is quickly verified that the Gaussian functions $f_j(x)=e^{-|x|^2}$ exhibit equality in \eqref{bl}; here $|\cdot|$ is the Euclidean norm on $H_j$. 

\begin{comment}
A Brascamp--Lieb inequality is a bound on the map $$(f_1,\hdots,f_m)\mapsto (f_1\circ B_1,\hdots,f_m\circ B_m),$$ formed by pulling back $f_j$ by $B_j$ for each $j$. The adjoint Brascamp--Lieb inequalities are bounds on the map
\begin{equation}\label{entro}
f\mapsto ((B_1)_*f,\hdots,(B_m)_*f),
\end{equation}
where, for each $1\leq j\leq m$, the function $(B_j)_*f:H_j\rightarrow \mathbb{R}_+$ is the pushforward of a function $f:H\rightarrow\mathbb{R}_+$ by $B_j$. 
Inequalities for \eqref{entro} were initially obtained in terms of entropies in \cite{CC}, generalising the classical entropy subadditivity inequality of Shearer.
In the context of Lebesgue spaces, a
\end{comment}
Brascamp--Lieb inequalities also have adjoints. These take the form
\begin{equation}\label{abl}
    \|f\|_{L^p(H)}\leq \ABL(\mathbf{B},\mathbf{c},\mathbf{\theta},p)\prod_{j=1}^m\|(B_j)_*f\|_{L^{p_j}(H_j)}^{\theta_j},
\end{equation}
where for each $1\leq j\leq m$, the function $(B_j)_*f:H_j\rightarrow \mathbb{R}_+$ is the pushforward of a function $f:H\rightarrow\mathbb{R}_+$ by $B_j$. 
Here the exponents $p,p_1,\hdots,p_m\in (0,1]$ and $\mathbf{\theta}=(\theta_1,\hdots,\theta_m)\in (0,1]^m$ are related to the Brascamp--Lieb exponents $\mathbf{c}$ via the compatibility condition
\begin{equation}\label{ablexponents}
    c_j\left(1-\frac{1}{p}\right)=\theta_j\left(1-\frac{1}{p_j}\right)
\end{equation}
for each $j$,
along with the (necessary) homogeneity condition $\theta_1+\cdots +\theta_m=1$; see \cite{BT}. Evidently the adjoint Brascamp--Lieb inequality \eqref{abl} incorporates the adjoint Finner inequality \eqref{af}, where $\ABL(\mathbf{B},\mathbf{c},\mathbf{\theta},p)=1$ for such data.
It was shown in \cite{BT} that \begin{equation}\label{btmain}
\ABL(\mathbf{B},\mathbf{c},\mathbf{\theta},p)\leq \BL(\mathbf{B},\mathbf{c})^{\frac{1}{p}-1}
\end{equation}
in general,
so that in particular $\ABL(\mathbf{B},\mathbf{c},\mathbf{\theta},p)\leq 1$ whenever the data $(\mathbf{B},\mathbf{c})$ is geometric.

As discussed in Section \ref{Sect:preliminaries}, Theorem \ref{thm:main} would follow if one could replace the use of Finner's inequality in the proof of Theorem \ref{thm:BTlog} with a different Brascamp--Lieb inequality whose adjoint constant is strictly less than 1. To find a suitable Brascamp--Lieb inequality with adjoint constant less than $1$ we appeal to some recent progress on stable forms of the Brascamp--Lieb inequality.
Following \cite[Section 7]{BT}, a Brascamp--Lieb datum is said to be \textit{Gaussian-stable} (or \textit{positive Gaussian-stable}, following \cite{BC:stable}) if $\BL(\mathbf{B},\mathbf{c})<\infty$ and for every $\varepsilon>0$ there exists $\delta>0$ such that $$\dist_1(f_j,\mathfrak{G}_{\mathbb{R}_+}(H_j)):=\inf_{g\in\mathfrak{G}_{\mathbb{R}_+}(H_j)}\|f_j-g\|_{L^1(H_j)}\leq\varepsilon\|f_j\|_{L^1(H_j)}$$ for each $1\leq j\leq m$ whenever the inputs $f_1,\hdots, f_m$ satisfy
$$
\int_{H}\prod_{j=1}^m f_j(B_jx)^{c_j}\mathrm{d}x\geq (1-\delta)\BL(\mathbf{B},\mathbf{c})\prod_{j=1}^m\left(\int_{H_j}f_j\right)^{c_j}.
$$
Here $\mathfrak{G}_{\mathbb{R}_+}(H_j)$ denotes the set of Gaussian functions of the form $ce^{-Q(x)+x\cdot v}$, where $Q$ is a symmetric positive-definite quadratic form on $H_j$, $v\in H_j$ and $c>0$.
The key feature of Gaussian stability is that it produces improved bounds for adjoint Brascamp--Lieb constants.
\begin{theorem}[\cite{BT}]
If $(\mathbf{B},\mathbf{c})$ is Gaussian-stable then there is strict inequality in \eqref{btmain} whenever $p<1$.
\end{theorem}
Consequently, any Gaussian stable geometric Brascamp--Lieb datum satisfies
$$\ABL(\mathbf{B},\mathbf{c},\mathbf{\theta},p)<1$$
whenever $p<1$. Arguing as in our proof of Theorem \ref{thm:BTlog} in Section \ref{Sect:preliminaries}, we obtain the following conditional statement that yields the conclusion of Theorem \ref{thm:main}.
\begin{theorem}\label{maincond}
    Suppose that there exists a Gaussian-stable co-rank-$k$ geometric Brascamp--Lieb datum $(\mathbf{B},\mathbf{c})$ on $\mathbb{R}^d$ with $c_1=\cdots =c_m$. If $0<p,q\leq 1$ satisfy the scaling condition \eqref{scalingk}, then for all $p<1$ there exists $D=D(p,k,d)>0$ such that 
\begin{equation}\label{ineqmain2}
\log\|f\|_{L^p(\mathbb{R}^d)}\leq -D+\int_{\mathcal{G}_{k,d}}\log\|T_{k,d}f(\pi,\cdot)\|_{L^q(\pi^\perp)}\mathrm{d}\pi
    \end{equation} 
    for all measurable functions $f:\mathbb{R}^d\rightarrow\mathbb{R}_+$.
\end{theorem}
As we have seen in the proof of Theorem \ref{thm:BTlog}, it is not at all difficult to exhibit a co-rank-$k$ geometric Brascamp--Lieb datum in $\mathbb{R}^d$ for any $k$ and $d$. There are many instances of Finner's inequality that have this property. The difficulty lies in finding examples that are Gaussian-stable, although one might conjecture that this is also always possible. For this we shall appeal to the following sufficient condition for Gaussian-stability, established recently by Christ and the first author.
\begin{theorem}[\cite{BC:stable}]\label{thm:BC}
If $\frac{1}{2}<c_j<1$ for all $1\leq j\leq m$, then $(\mathbf{B},\mathbf{c})$ is Gaussian-stable.
\end{theorem}
Combining this theorem with the following proposition allows us to recover certain cases of Theorem \ref{thm:main} via Theorem \ref{maincond}.
\begin{proposition}\label{prop:tff}
    Suppose that $d/2\leq k\leq d-1$, and that $2k-d$ divides $d$ if $2/5<(d-k)/d<1/2$. Then there exists a co-rank-$k$ geometric Brascamp--Lieb datum $(\mathbf{B},\mathbf{c})$ on $\mathbb{R}^d$ satisfying $\frac{1}{2}<c_1=\cdots= c_m<1$.
\end{proposition}
    In the special case $d=2k$, such an example of a Brascamp--Lieb datum is already provided by Young's convolution inequality
$$
\int_{\mathbb{R}^{k}}\int_{\mathbb{R}^k}f(x)^{2/3}g(y)^{2/3}h(x+y)^{2/3}\mathrm{d}x\mathrm{d}y\lesssim\left(\int_{\mathbb{R}^k}f\right)^{2/3}\left(\int_{\mathbb{R}^k}g\right)^{2/3}\left(\int_{\mathbb{R}^k}h\right)^{2/3},
$$
which being Gaussian-extremisable is equivalent to a co-rank-$k$ geometric Brascamp--Lieb inequality (see \cite[Proposition 3.6]{BCCT}) with exponents $c_1=c_2=c_3=\frac{2}{3}$. By considering an iterated convolution, a similar conclusion may be reached whenever $d-k$ divides $d$. This condition is stronger than the hypothesis of Proposition \ref{prop:tff}, as may be seen with the example $(d,k)=(5,3)$. 
In the next subsection we establish Proposition \ref{prop:tff} in general by first interpreting it as a statement in frame theory. The key ingredient then becomes the characterisation of tight fusion frames established in \cite{CFMWZ}. 
\begin{remark}
Gaussian-stability was established in the special case of Young's convolution inequality in \cite{C:stable Y}, without the assumption that $1/2<c_1,c_2,c_3<1$. 
\end{remark}

\begin{remark}
    At present, this approach, based on the adjoint Brascamp--Lieb constant, establishes Theorem \ref{thm:main} only under the assumption in Proposition \ref{prop:tff}. Its advantage, however, is that it naturally yields quantitative bounds, since  $D(p,k,d)\ge-\log \ABL$. 
    %Moreover, the optimal constant $D(p,k,d)$ may be interpreted as (the logarithm of) an adjoint Brascamp--Lieb constant in the setting of the \textit{continuous} Brascamp--Lieb data of Barthe \cite{Barthe cts}. 
\end{remark}

\subsection{Proof of Proposition \ref{prop:tff}}\label{Sect:fusion} 

We begin by observing that the geometric Brascamp--Lieb condition \eqref{gblframe}, when specialised to coincident exponents $c_1=\cdots=c_m$, is a \textit{frame condition}. The question of the existence of such frames is important in frame theory, to which we now turn. 

Let $P_j:\mathbb R^d\rightarrow \mathbb R^d$ be a projection whose image is $\pi_j^\perp$, where $\pi_j$ is a $k$-dimensional subspace of $\mathbb R^d$, which says that $\pi_j^\perp$ is a $(d-k)$-dimensional subspace.
For an integer $n\ge1$, we say that a collection $\{P_1,\hdots,P_n\}$ is a \textit{$(n,d-k,d)$-tight fusion frame} if
\[
\sum_{j=1}^{n}P_j=\frac{n(d-k)}{d}I.
\]
Thus, the existence of an $(n,d-k,d)$-tight fusion frame directly implies the existence of a geometric Brascamp--Lieb datum $(\mathbf B,\mathbf c)$ satisfying $B_j^*B_j=P_j$ and $c_j=d/(n(d-k))$. Consequently, Proposition \ref{prop:tff} reduces to verifying the existence of an $(n,d-k,d)$-tight fusion frame satisfying $n(d-k)/d\in(1,2)$ under the assumptions of Proposition \ref{prop:tff}.

In \cite{CFMWZ}, the authors completely characterise the tuples $(n,d-k,d)$ for which an $(n,d-k,d)$-tight fusion frame exists in the complex setting (i.e. where $\mathbb{R}$ is replaced with $\mathbb{C}$). As commented in \cite[Remark 1.16]{NS}, by \cite[Theorem 20]{Fulton}, it also provides a complete characterisation in the required real setting.
\begin{theorem}[\cite{CFMWZ}]\label{thm:CFMWZ1}
    Let $n,k,d$ be positive integers satisfying $k<d$. Then the following hold.
    \begin{enumerate}
        \item Assume that $d-k$ divides $d$. Then there exists an $(n,d-k,d)$-tight fusion frame if and only if $n\ge d/(d-k)$.
        \item Assume that $d-k$ does not divide $d$ and $d-k<d/2$. If there exists an $(n,d-k,d)$-tight fusion frame, then $n\ge d/(d-k)+1$.
        \item Assume that $d-k$ does not divide $d$ and $d-k<d/2$. If $n\ge d/(d-k)+2$, then there exists an $(n,d-k,d)$-tight fusion frame.
    \end{enumerate}
\end{theorem}
The authors also provide a reduction procedure that allows one to treat certain cases by reducing them to one of those above.
\begin{theorem}[\cite{CFMWZ}]\label{thm:CFMWZ2}
Let $n,k,d$ be positive integers satisfying $k<d$ and $k+d<nk$. The following are equivalent:    
\begin{enumerate}
    \item There exists an $(n,k,d)$-tight fusion frame.
    \item There exists an $(n,d-k,d)$-tight fusion frame.
    \item There exists an $(n,(n-1)k-d,nk-d)$-tight fusion frame.
\end{enumerate}
\end{theorem}

Using the above Theorems, for given integers $n,k,d$, we can check for the existence of an $(n,d-k,d)$-tight fusion frame.  From now on, for simplicity, we say $(n,k,d)\in \mathrm{TFF}$ if there exists an $(n,k,d)$-tight fusion frame. By the above theorems, one can easily check the following corollary, which proves Proposition \ref{prop:tff} for the cases $d/(d-k)\ge 5/2$ or $d=2(d-k)$.
\begin{corollary}\label{cor:CFMWZ}
    Let $d/(d-k)\ge5/2$ or $d=2(d-k)$. Then there exists $n\in\mathbb N$ such that $(n,d-k,d)\in \mathrm{TFF}$ and $n(d-k)/d\in(1,2)$.
\end{corollary}
\begin{proof}
If $d-k$ divides $d$, the result follows immediately from Theorem \ref{thm:CFMWZ1} (1). Hence, from now on, we assume that $d-k$ does not divide $d$.

    First, consider the case $d/(d-k)>5/2$. Let $n_{k,d}$ be the smallest integer larger than $d/(d-k)+2$.
    Then, by Theorem \ref{thm:CFMWZ1} (3), we have $(n_{k,d},d-k,d)\in \mathrm{TFF}$. Moreover, the condition $d/(d-k)>5/2$ implies that $n_{k,d}(d-k)/d<2$ as desired.

    Now, suppose that $d/(d-k)=5/2$. Note that $5$ divides $d$. For this case, set $n=4$. Then, by Theorem \ref{thm:CFMWZ2}, $(4,d-k,d)\in \mathrm{TFF}$ is equivalent to $(4,d/5,3d/5)\in \mathrm{TFF}$. Thus, Theorem \ref{thm:CFMWZ1} (1) gives the desired result combined with $4(d-k)/d\in(1,2)$.
\end{proof}
By Corollary \ref{cor:CFMWZ}, if $(d-k)/d\le 2/5$ or $(d-k)/d=1/2$, there always exists an $(n,d-k,d)$-tight fusion frame satisfying $n(d-k)/d\in (1,2)$. However, for the remaining case $2/5<(d-k)/d<1/2$, such a frame may fail to exist for certain cases. This obstruction arises from the relatively large value of $(d-k)/d$. Indeed, since $(d-k)/d>2/5$, the condition $n(d-k)/d<2$ forces $n\le 4$. On the other hand, Theorem \ref{thm:CFMWZ1} (2) requires $n\ge4$. Thus, in this regime, the only possible choice for the desired frame is $n=4$. In this context, the following theorem provides a characterization of the existence of a $(4,d-k,d)$-tight fusion frame.

\begin{theorem}\label{thm:4ln}
    Let $2/5<L/N<1/2$ and set 
    \[
    m_*=[L/(N-2L)],
    \]
    where $[r]$ denotes the largest integer smaller than or equal to $r$.    Then $(4,L,N)\in \mathrm{TFF}$ if and only if either 
    \begin{itemize}
        \item $N-2L$ divides $L$, or
        \item $L-m_*(N-2L)$ divides $N-2m_*(N-2L)$ and $2L-2m_*(N-2L)\ge N-2L$.
    \end{itemize}
\end{theorem}
\begin{proof}
Set $M:=N-2L$.
    By Theorem \ref{thm:CFMWZ2}, we have
    \[
    (4,L,N)\in \mathrm{TFF} \Leftrightarrow (4,L-mM,N-2mM) \in \mathrm{TFF},
    \]
    for $0\le m< m_*$ if $M$ divides $L$, and for $0\le m \le m_*$ otherwise.

    If $M$ divides $L$, choosing $m=L/M-1$ yields that
    \[
    (4,L,N)\in \mathrm{TFF} \Leftrightarrow (4,M,3M) \in \mathrm{TFF}.
    \]
    From Theorem \ref{thm:CFMWZ1} (1), we conclude that $(4,L,N)\in \mathrm{TFF}$.

    Now, assume that $N-2L$ does not divide $L$. 
    First, consider the case when $L-m_*M$ does not divide $N-2m_*M$. Since
    \[
    \frac{N-2m_*M}{L-m_*M}=2+\frac{M}{L-m_*M}>3,
    \]
    Theorem \ref{thm:CFMWZ1} (2) shows that $(4,L-m_*M,N-2m_*M)\notin \mathrm{TFF}$, which implies $(4,L,N)\notin \mathrm{TFF}$. Secondly, assume that $L-m_*M$ divides $N-2m_*M$. Then, by Theorem \ref{thm:CFMWZ1} (1), $(4,L-m_*M,N-2m_*M)\in \mathrm{TFF}$ if and only if $4(L-m_*M)\ge N-2m_*M$, which is equivalent to 
    \begin{equation}
    \label{ineq:lmn}
    2(L-m_*M)\ge N-2L.        
    \end{equation}
    Thus, for this case, $(4,L,N)\in \mathrm{TFF}$ if and only if \eqref{ineq:lmn} holds, which is the desired result.
\end{proof}
Theorem \ref{thm:4ln} gives a necessary and sufficient condition for $(4,d-k,d)\in \mathrm{TFF}$. Although the condition in the theorem may appear complex, it can be reformulated in a simpler form.
\begin{corollary}
    Let $2/5<(d-k)/d<1/2$. Then there exists $n\in \mathbb N$ such that $(n,d-k,d)\in \mathrm{TFF}$ and $n(d-k)/d\in (1,2)$ if and only if $2k-d$ divides $2d-2k$ (if and only if $2k-d$ divides $d$).
    \end{corollary}
\begin{proof}
    It is enough to show that the following are equivalent:
    \begin{enumerate}
        \item $N-2L$ divides $2L$.
        \item either $N-2L$ divides $L$, or otherwise $L-m_*(N-2L)$ divides $N-2m_*(N-2L)$ and $2L-2m_*(N-2L)\ge N-2L$.
    \end{enumerate}

    First, show that (1) implies (2). Suppose that (1) holds. Then either $N-2L$ divides $L$, or
    \[
    L=m_*(N-2L)+\frac{N-2L}{2},
    \]
    for some $m_*\ge0$. For the latter, we have $L-m_*(N-2L)=(N-2L)/2$ divides $N-2m_*(N-2L)=N-2L+2(L-m_*(N-2L))$, and $2L-2m_*(N-2L)=N-2L$. Thus, (2) follows by the fact $m_*=[L/(N-2L)]$.

    Now, prove that (2) yields (1). It is trivial if $N-2L$ divides $L$. From now on, we assume that (2) holds and $N-2L$ does not divide $L$. Set
    \[
    L=m_*(N-2L)+l
    \]
    where $0\le l<N-2L$. Then the assumption can be reformulated by
    $l$ divides $N-2L +2l$ and $2l\ge N-2L$, which means that $l=(N-2L)/2$. Thus, $N-2L$ divides $2L$.
\end{proof}

\begin{remark}
    Let $d-k\le d/2$. If $d\le 16$ or $d-k\le 6$, then there exists $n$ such that $(n,d-k,d)\in \mathrm{TFF}$ and $1<n(d-k)/d<2$. Note that for $(d-k,d)=(7,17)$, there is no $(n,7,17)$-tight fusion frame satisfying $1<n(d-k)/d<2$.
\end{remark}

\section{Adjoint continuous Brascamp--Lieb inequalities}\label{Sect:cts}
%Explain here the interpretation of $D$ as an adjoint continuous Brascamp--Lieb constant. This is essentially in the earlier set of notes, although that is written down in the co-rank-1 setting only. WHAT REMAINS IS AN EARLIER VERSION OF MY NOTES THAT WE CAN MAYBE USE AS A BASE FOR SECTION \ref{Sect:proof}.
The purpose of this section is to provide an interpretation of the positive constant $D(p,k,d)$ from Theorem \ref{thm:main} in terms of the \textit{continuous} Brascamp--Lieb inequalities introduced by Barthe \cite{Barthe cts}. The point is simply that the inclusion of a continuum of projections in the proof of Theorem \ref{thm:main} removes the need for the lossy averaging step. One then sees a direct relationship between $D(p,k,d)$ and the relevant adjoint (continuous) Brascamp--Lieb constant. This perspective was suggested in \cite[Remark 8.3]{BT}.

Let us clarify first the set-up for continuous geometric Brascamp--Lieb inequalities.
By analogy with the usual (geometric Brascamp--Lieb, see \cite{BCCT} for example) set-up, we suppose that we have a measure $\mu$ on $\mathcal{G}_{k,d}$ satisfying an analogue of \eqref{gblframe},
$$
\int_{\mathcal{G}_{k,d}}P_{\pi^\perp}^*P_{\pi^\perp}\mathrm{d}\mu(\pi)=I_d.
$$
%such measures are called \textit{isotropic}.\footnote{I'm not sure where this is written. I thought it was in \cite{Barthe_cts}, but after another quick look I got the impression that they deal with rank-1 projections only -- presumably it is somewhere, perhaps in \cite{CM}. It is instructive to consider a measure $\mu$ that is a sum of suitably normalised dirac masses at the points $e_1,\hdots,e_n$, as one then has classical Loomis--Whitney data.}
Notice that by taking traces we have 
\begin{equation}\label{mass}
\mu(\mathcal{G}_{k,d})=\frac{d}{d-k}.
\end{equation}
The corresponding geometric Brascamp--Lieb inequality is
\begin{equation}\label{contLW}
\int_{\mathbb{R}^d}\exp\left(\int_{\mathcal{G}_{k,d}}\log f_\pi(P_{\pi^\perp} x)\mathrm{d}\mu(\pi)\right)\mathrm{d}x\leq\exp\left(\int_{\mathcal{G}_{k,d}}\log\left(\int_{\pi^\perp}f_\pi\right)\mathrm{d}\mu(\pi)\right),
\end{equation}
holding for all integrable $f_\pi:\pi^\perp\rightarrow\mathbb{R}_+$, $\pi\in\mathcal{G}_{k,d}$; see \cite{Barthe cts} for the rank-one case, and \cite{CEM, ML} in general.
%The constant $1$ should be best-possible, and the gaussian functions $f_\omega(x)=e^{-\pi |x|^2}$ (for all $\omega$) should be among the extremisers. Presumably this can be proved by heat flow, just as in the usual setting of geometric Brascamp--Lieb, although I'm not sure if/where this is written explicitly. I am working on the assumption that \eqref{contLW} is correct, which I believe it is.

The adjoint inequalities also require some setting up. Let $0<p\leq 1$. Given a measure $\mu$, let $\nu$ be a probability measure on $\mathcal{G}_{k,d}$ for which $\mu$ and $\nu$ are mutually absolutely continuous. For $\pi\in\mathcal{G}_{k,d}$ let $0<p_\pi\leq 1$ be such that
\begin{equation}\label{scalcont}
    \left(\frac{1}{p_\pi}-1\right)\mathrm{d}\nu(\pi)=\left(\frac{1}{p}-1\right)\mathrm{d}\mu(\pi).
\end{equation}
Evidently the function $\pi\mapsto p_\pi$ is only defined almost everywhere (with respect to either measure). By analogy with \cite[Definition 1.6]{BT}, the associated adjoint Brascamp--Lieb inequality then takes the form
\begin{equation}\label{contadjLWgendef}
	\log\|f\|_{L^p(\mathbb{R}^d)}\leq \log \ABL+\int_{\mathcal{G}_{k,d}}\log\left(\|(P_{\pi^\perp})_*f\|_{L^{p_\pi}(\pi^\perp)}\right)\mathrm{d}\nu(\pi),
\end{equation}
where $\ABL$ is the least constant for which \eqref{contadjLWgendef} holds for all nonnegative $f\in L^p(\mathbb{R}^d)$. We refer to $\ABL$ as the adjoint (continuous) Brascamp--Lieb constant. Although this depends on the data $p, p_\pi, \mu, \nu$, it is straightforward to see that $\ABL\leq 1$ in general.
\begin{theorem}[cf. \cite{BT}, Theorem 1.11]\label{thm:aclw} Assume that \eqref{contLW} holds. Then $\ABL\leq 1$, that is,
\begin{equation}\label{contadjLWgen}
    \log\|f\|_{L^p(\mathbb{R}^d)}\leq\int_{\mathcal{G}_{k,d}}\log\left(\|(P_{\pi^\perp})_*f\|_{L^{p_\pi}(\pi^\perp)}\right)\mathrm{d}\nu(\pi),
\end{equation}
for all nonnegative functions $f$ on $\mathbb R^d$.
\end{theorem}
Taking $\mathrm{d}\mu$ to be uniform on $\mathcal{G}_{k,d}$, one is naturally led to the special case in which the map $\pi\rightarrow p_\pi$ is constant and $d\nu$ is the unique $O(d)$-invariant probability measure on $\mathcal{G}_{k,d}$. Then \eqref{mass} and \eqref{scalcont} imply \eqref{scalingk} with $p_\pi\equiv q$. Consequently, inequality \eqref{contadjLWgen} may be viewed as a generalisation of \eqref{Application of Finnerintro}. Moreover, the optimal constant $D(p,k,d)$ in Theorem \ref{thm:main} is precisely $-\log\ABL$. Consequently, Theorem \ref{thm:main} establishes that $\ABL<1$ for this \textit{uniform} continuous Brascamp--Lieb inequality.

For completeness we conclude this section with a proof of Theorem \ref{thm:aclw}. This follows closely that of \cite[Theorem 1.11]{BT}.
\begin{proof}
    We argue as in \cite{BT}, although using the continuous version of H\"older's inequality, which in this context states that
    \begin{equation}\label{contHolder}
        \int_{\mathbb{R}^d}\exp\left(\int_{\mathcal{G}_{k,d}}\log (g_\pi(x))\mathrm{d}\nu(\pi)\right)\mathrm{d}x\leq\exp\left(\int_{\mathcal{G}_{k,d}}\log\left(\int_{\mathbb{R}^d}g_\pi(x)\mathrm{d}x\right)\mathrm{d}\nu(\pi)\right);
    \end{equation}
    the H\"older scaling condition is captured by the fact that $\mathrm{d}\nu$ is a probability measure.
    For each $\pi\in\mathcal{G}_{k,d}$, let $f_\pi:=(P_{\pi^\perp})_*f:\pi^\perp\rightarrow\mathbb{R}_+$ be the pushforward of $f$ by the linear map $P_{\pi^\perp}$. By definition, the pullback-pushforward duality relation holds,
    \begin{equation}\label{duality}
\int_{\pi^\perp}f_\pi \varphi:=\int_{\mathbb{R}^d}f(x)\varphi(P_{\pi^\perp} x)\mathrm{d}x.
    \end{equation}
    
    We begin by using that $\mathrm{d}\nu$ is a probability measure again to observe that $\|f\|_p^p$ is equal to the following:
    \begin{eqnarray*}
        \begin{aligned}
            &\int_{\mathbb{R}^d}f(x)^p\exp\left(p\int_{\mathcal{G}_{k,d}}\log(f_\pi(P_{\pi^\perp} x)^{p_\pi-1})\mathrm{d}\nu(\pi)\right)\exp\left(p\int_{\mathcal{G}_{k,d}}\log(f_\pi(P_{\pi^\perp} x)^{1-p_\pi})\mathrm{d}\nu(\pi)\right)\mathrm{d}x\\
            &=\int_{\mathbb{R}^d}\exp\left(p\int_{\mathcal{G}_{k,d}}\log(f(x)f_\pi(P_{\pi^\perp} x)^{p_\pi-1})\mathrm{d}\nu(\pi)\right)\exp\left(p\int_{\mathcal{G}_{k,d}}\log(f_\pi(P_{\pi^\perp} x)^{1-p_\pi})\mathrm{d}\nu(\pi)\right)\mathrm{d}x.
            %&\leq\left(\int_{\mathbb{R}^n}\exp\left(\log
        \end{aligned}
    \end{eqnarray*}
    By the usual H\"older inequality on $\mathbb{R}^d$ with exponents $1/p$ and $1/(1-p)$, we have
    \begin{eqnarray*}
        \begin{aligned}
            \|f\|_p^p&\leq\left(\int_{\mathbb{R}^d}\exp\left(\int_{\mathcal{G}_{k,d}}\log(f(x)f_\pi(P_{\pi^\perp} x)^{p_\pi-1})\mathrm{d}\nu(\pi)\right)\mathrm{d}x\right)^p\\&\times\left(\int_{\mathbb{R}^d}\exp\left(\frac{p}{1-p}\int_{\mathcal{G}_{k,d}}\log(f_\pi(P_{\pi^\perp} x)^{1-p_\pi})\mathrm{d}\nu(\pi)\right)\mathrm{d}x\right)^{1-p}.
            %&\leq\left(\int_{\mathbb{R}^n}\exp\left(\log
        \end{aligned}
    \end{eqnarray*}
    We apply \eqref{contHolder} to the first factor here, and \eqref{contLW} to the second. Beginning with the first, we have
    \begin{eqnarray}\label{restrictedholder}
        \begin{aligned}
            \int_{\mathbb{R}^d}\exp&\left(\int_{\mathcal{G}_{k,d}}\log(f(x)f_\pi(P_{\pi^\perp} x)^{p_\pi-1})\mathrm{d}\nu(\pi)\right)\mathrm{d}x \\            
            &\leq\exp\left(\int_{\mathcal{G}_{k,d}}\log\left(\int_{\mathbb{R}^d}f(x)f_\pi(P_{\pi^\perp} x)^{p_\pi-1}\mathrm{d}x\right)\mathrm{d}\nu(\pi)\right)\\
            &=\exp\left(\int_{\mathcal{G}_{k,d}}\log\left(\int_{\pi^\perp}f_\pi^{p_\pi}\right)\mathrm{d}\nu(\pi)\right),
        \end{aligned}
    \end{eqnarray}
    by \eqref{duality}.
    For the second factor, we use \eqref{scalcont} and some algebraic manipulation to write
    \begin{eqnarray*}
        \begin{aligned}
            \int_{\mathbb{R}^d}\exp\left(\frac{p}{1-p}\int_{\mathcal{G}_{k,d}}\log(f_\pi(P_{\pi^\perp} x)^{1-p_\pi})\mathrm{d}\nu(\pi)\right)\mathrm{d}x=\int_{\mathbb{R}^d}\exp\left(\int_{\mathcal{G}_{k,d}}\log(f_\pi(P_{\pi^\perp} x)^{p_\pi})\mathrm{d}\mu(\pi)\right)\mathrm{d}x
        \end{aligned}
    \end{eqnarray*}
    and so by \eqref{contLW},
    \begin{eqnarray*}
        \begin{aligned}
            \int_{\mathbb{R}^d}\exp\left(\frac{p}{1-p}\int_{\mathcal{G}_{k,d}}\log(f_\pi(P_{\pi^\perp} x)^{1-p_\pi})\mathrm{d}\nu(\pi)\right)\mathrm{d}x\leq\exp\left(\int_{\mathcal{G}_{k,d}}\log\left(\int_{\pi^\perp}f_\pi^{p_\pi}\right)\mathrm{d}\mu(\pi)\right).
        \end{aligned}
    \end{eqnarray*}
    Putting these two estimates together we have
    \begin{eqnarray*}
        \begin{aligned}
            \|f\|_p^p\leq \exp\left(p\int_{\mathcal{G}_{k,d}}\log\left(\int_{\pi^\perp}f_\pi^{p_\pi}\right)\mathrm{d}\nu(\pi)\right)\exp\left((1-p)\int_{\mathcal{G}_{k,d}}\log\left(\int_{\pi^\perp}f_\pi^{p_\pi}\right)\mathrm{d}\mu(\pi)\right).
        \end{aligned}
    \end{eqnarray*}
    Finally we apply \eqref{scalcont} to change the measure $\mathrm{d}\mu$ to the required probability measure $\mathrm{d}\nu$, and the theorem follows.
\end{proof}

\section{Use of AI}
ChatGPT Sol was used to sharpen the constant in Lemma \ref{lem:P1P2}, which was initially established with a suboptimal constant by the authors using an ad-hoc argument. While the sharp constant is not needed in this paper, it is included for its interesting short proof. All other arguments are due to the authors.  ChatGPT was also used to read through the paper looking for errors and oversights on completion. This led to some typographical corrections, and prompted us to elaborate on Remark \ref{Rem:EFKY} in Section \ref{Sect:stabmeas}.

\end{document}